\documentclass[11pt,english,a4paper]{smfart}
\usepackage[english]{babel}
\usepackage{xcolor}
\usepackage{stmaryrd}
 \usepackage{enumerate}
\usepackage{amsmath,amssymb,bm,amsfonts}
\usepackage[all]{xy}
\entrymodifiers={+!!<0pt,\fontdimen22\textfont2>}

\usepackage{mathrsfs}

\usepackage{graphicx}

\usepackage{color}

\newtheorem{theorem}{Theorem}[section]
\newtheorem{lemma}[theorem]{Lemma}

\newtheorem{corollary}[theorem]{Corollary}
\newtheorem{proposition}[theorem]{Proposition}

\newtheorem*{remark*}{\it Remark}

\renewcommand*{\thesubsection}{\thesection\alph{subsection}}

\newcommand{\fl}[2]{
\xymatrix@C15pt{#1\ar[r]&#2}}
\newcommand{\flba}[2]{
\xymatrix@C15pt{#1\ar@{|->}[r]&#2}}
\newcommand{\flcourte}[2]{
\xymatrix@C12pt{#1\ar[r]&#2}}

{\theoremstyle{definition}}

{\theoremstyle{definition}\newtheorem{example}[theorem]{Example}}

\theoremstyle{definition}
\newtheorem{definition}[theorem]{Definition}
\newtheorem{question}[theorem]{Question}
\newtheorem{fact}[theorem]{Fact}
\newtheorem{claim}[theorem]{Claim}

{\theoremstyle{definition}\newtheorem{remark}[theorem]{Remark}}

\def\T{\ensuremath{\mathbb T}}
\def\R{\ensuremath{\mathbb R}}
\def\Z{\ensuremath{\mathbb Z}}

\def\C{\ensuremath{\mathbb C}}
\def\Q{\ensuremath{\mathbb Q}}
\def\N{\ensuremath{\mathbb N}}
\def\P{\ensuremath{\mathbb P}}

\newcommand{\pss}[2]{\ensuremath{{\langle #1,#2\rangle}}}

\newcommand{\wh}[1]{\widehat{#1}}

\newcommand{\ds}{\displaystyle}
\newcommand{\ba}[1]{\overline{#1}}

\newcommand{\ka}{Kazhdan}

\newcommand{\rnk}{\mathcal R^{(n_k)}}
\newcommand{\Conv}{\mathop{\scalebox{1.5}{\raisebox{-0.2ex}{$\ast$}}}}%

\newcommand{\sbt}{\,\begin{picture}(-1,0)(-1,-2)\circle*{3}\end{picture}\ }

\newcommand{\nk}{n_{k}}
\newcommand{\nkp}[1]{(n_{k})_{k\ge #1}}
\newcommand{\mk}{m_{k}}
\newcommand{\much}{\widehat{\mu}}

\newcommand{\mpt}{\mathcal{P}(\T)}

\newcommand{\ep}{\varepsilon}
\DeclareMathOperator{\dist}{dist}
\numberwithin{equation}{section}

\author[C. Badea]{Catalin Badea}
\address[C. Badea]{Univ. Lille, CNRS, UMR 8524 - Laboratoire Paul Painlev\'{e}, F-59000 Lille, France}
\email{catalin.badea@univ-lille.fr}
\urladdr{http://math.univ-lille1.fr/~badea/}
\author[S. Grivaux]{Sophie Grivaux}
\address[S. Grivaux]{CNRS, Univ. Lille, UMR 8524 - Laboratoire Paul Painlev\'{e}, F-59000 Lille, France}
\email{sophie.grivaux@univ-lille.fr}
\urladdr{http://math.univ-lille1.fr/~grivaux/}
\author[\'{E}. Matheron]{\'{E}tienne Matheron}
\address[\'{E}. Matheron]{Laboratoire de Math\'{e}matiques de Lens, Universit\'{e} d'Artois, Rue Jean Souvraz SP 18, F-62307 Lens, France}
\email{etienne.matheron@univ-artois.fr}
\urladdr{http://matheron.perso.math.cnrs.fr/}

\date{\today}

\begin{document}

\title[Rigidity sequences, Kazhdan sets and group topologies]{Rigidity sequences, Kazhdan sets and\\
group topologies on the integers}

 \keywords{Fourier coefficients of continuous measures; rigidity sequences; Kazhdan subsets of $\mathbb{Z}$ and Kazhdan sequences; group topologies; nullpotent sequences; additive basis of 
 $\mathbb{Z}$; random subsets of $\mathbb{Z}$; the probabilistic method}
  \subjclass{11J71, 22D40, 28D05, 37A25, 42A16, 42A61, 43A46, 47A10, 47A35}
 \thanks{This work was supported in part by
 the project FRONT of the French
National Research Agency (grant ANR-17-CE40-0021) and by the Labex CEMPI (ANR-11-LABX-0007-01)}

\begin{abstract}
We study the relationships between three different classes of sequences (or sets) of integers, 
namely rigidity sequences, Kazhdan sequences (or sets) and nullpotent sequences. We prove that rigidity sequences are non-Kazhdan and nullpotent, and that all other implications are false. In particular, we show by probabilistic means that there exist sequences of integers which are both nullpotent and Kazhdan. Moreover, using Baire category methods, we provide general criteria for a sequence of integers to be a rigidity sequence. Finally, we give a new proof of the existence of rigidity sequences which are dense in $\Z$ for the Bohr topology, a result originally due to Griesmer.
\end{abstract}

\maketitle
\section{Introduction}\label{Introduction} This paper is centered around three classes of  sequences (or subsets) of $\Z$, the additive group of the integers. These are rigidity sequences, Kazhdan sequences (or sets) and nullpotent sequences, \mbox{\it i.e.} sequences which converge to zero with respect to some Hausdorff group topology on $\Z$. The motivation for considering these sequences is that they appear naturally in the study of problems relevant to harmonic analysis, geometric group theory, dynamical systems, and number theory. 

\subsection{Kazhdan sets}  \ka\ subsets of $\Z$ are defined as follows.

\begin{definition}\label{Def 1}
A subset $Q$ of $\Z$ is called a \emph{\ka\ set} if there exists $\varepsilon >0$ such that any unitary operator $U$ on a complex separable Hilbert space $H$ satisfies the following property: if there exists a vector $x\in H$ with $||x||=1$ such that 
$\sup_{n\in Q}||U^{n}x-x||<\varepsilon $, then $1$ is an eigenvalue of $U$, i.e.\ there exists a non-zero vector $y\in H$ such that $Uy=y$.
\end{definition}

The set  
$\Z$ itself is a \ka\ set in $\Z$ 
(this can be proved for instance by using von Neumann's mean ergodic theorem). 
On the other hand, by considering rotations on $H=\C$, it is easy to see that no finite subset of $\Z$ can be Kazhdan. This also follows from the fact that $\Z$ is a non-compact amenable group, and so it does not possess Kazhdan's property (T). Indeed, the
notion of \ka\ set actually makes sense in any topological group $G$ -- just replace unitary operators $U$ by (strongly continuous) unitary representations $\pi$ of $G$ -- and one of the equivalent definitions of Kazhdan's Property (T) is that a topological group $G$ has Property (T)  if and only if it admits a \emph{compact} \ka\ set. See the book \cite{BdHV} for more on Property (T) and its many applications to various fields. 

Even though Property (T) involves compact \ka\ sets, it was suggested in 
\cite[p. 284]{BdHV} that it is also of interest to study \ka\ sets in groups which do \emph{not} have Property (T). This topic is addressed in \cite{BG1}, where a characterization of generating \ka\ sets in second-countable locally compact groups is obtained. (A subset of a group $G$ is said to be \emph{generating} if it generates $G$ in the group-theoretic sense; in the case $G=\Z$, a subset of $\Z$ is generating if it is not contained in $p\Z$ for any $p\ge 2$.) This leads to an equidistribution criterion implying that a set is \ka, as well as to explicit characterizations of \ka\ sets in many groups without Property (T), such as locally compact abelian groups or Heisenberg groups. See also \cite{Ch} for examples of \ka\ sets in other Lie groups, and \cite{BaGrLyons} for dynamical applications.
\par\smallskip
In the present paper we will be interested in \ka\ subsets of $\Z$ only. One pleasant thing when working with the group $\Z$ is that the property of being or not a \ka\ subset of $\Z$ can be expressed  in terms of Fourier coefficients of probability measures on the unit circle $\T=\{z \in\C\,;\,|z |=1\}$. We denote by $\mathcal{P}(\T)$ the set of all Borel probability measures on $\T$, and we endow it with the Prokhorov topology (\textit{i.e.}\ the topology of weak convergence of measures), which turns it into a compact metrizable space. A measure $\mu \in\mathcal{P}(\T)$ is said to be \emph{continuous}, or \emph{atomless}, if $\mu (\{z \})=0$ for any $z\in\T$. The set of all continuous measures $\mu\in\mathcal P(\T)$ will be denoted by $\mathcal P_c(\T)$.
Particularizing some results of \cite{BG1}, we have the following characterizations (see also \cite{BG2} for direct proofs using tools from harmonic analysis).
\begin{theorem}
\label{Th 1}
 A subset $Q$ of $\Z$ is a \ka\ set if and only if there exists $\varepsilon >0$ such that the following property holds true:
\begin{enumerate}
\item[\emph{$(1)_\varepsilon$}] if $\mu \in\mpt$ is such that $\sup _{n\in Q}|\much(n)-1|<\varepsilon $, then $\mu (\{1\})>0$.
\end{enumerate}
Moreover, if $Q$ is a generating subset of $\Z$, it is equivalent to say that the following holds true for some $\varepsilon>0$:
\begin{enumerate}
\item[\emph{$(1')_\varepsilon$}] if $\mu \in\mpt$ is such that $\sup _{n\in Q}|\much(n)-1|<\varepsilon $, then $\mu $ has an atom.
\end{enumerate}

\smallskip\noindent
In other words, $Q$ is \emph{not} a \ka\ set if and only if
\begin{enumerate}
\item[\emph{(2)}] for every $\varepsilon >0$, there exists $\mu \in\mpt$ with $\mu (\{1\})=0$ such that $\sup_ {n\in Q}|\much(n)-1|<\varepsilon $;
\end{enumerate}
and if $Q$ is generating this is equivalent to
\begin{enumerate}
\item[\emph{(2')}] for every $\varepsilon >0$, there exists  $\mu \in {\mathcal{P}_c(\T)}$ such that $\sup _{n\in Q}|\much(n)-1|<\varepsilon $.
\end{enumerate}
\end{theorem}

Note that the necessity of condition $(1)_\varepsilon$ for some $\varepsilon>0$ is clear: indeed, $(1)_\varepsilon$ for a given measure $\mu\in\mathcal P(\T)$ is just the condition appearing in the definition of a Kazhdan set for the unitary operator $U$ defined on $L^2(\mu)$ by $Uf(z)=zf(z)$, $f\in L^2(\mu)$. Note also that $(1)_\varepsilon$ is easily seen to imply that $Q$ is a generating subset of $\Z$ (if $Q\subseteq p\Z$ for some $p\ge 2$, consider the measure $\mu:=\delta_{e^{2i\pi/p}}$); so any Kazhdan subset of $\Z$ is generating. More generally, Kazhdan subsets with non-empty interior in a locally compact group are generating.

\smallskip It will be more convenient for us to speak of Kazhdan \emph{sequences} rather than Kazhdan sets. Of course, a sequence $(n_k)_{k\geq 0}$ of elements of $\Z$ is said to be Kazhdan if the set $Q=\{ n_k;\; k\geq 0\}$ is Kazhdan.  
Here are examples of \ka\ sequences: $n_{k}:=k$, $n_{k}:=k^{2}$, or 
$n_{k}:=p(k)$ where $p\in\Z[X]$ is a non-constant polynomial such that the integers $p(k)$, $k\geq 0$, have no non-trivial common divisor.  More generally, any generating sequence $(n_{k})_{k\ge 0}$  such that $(n_{k}\theta )_{k\ge 0}$ is uniformly distributed $\bmod$ 1 for any $\theta \in\R\setminus \Q$ is a \ka\ sequence (a classical reference for uniform distribution is \cite{KuiNie}). This result, which 
gives an answer to a question of Shalom (\cite[Question 7.12]{BdHV}),  follows from Theorem \ref{Th 1}. 
On the other hand, if $n_k>0$ and ${n_{k+1}/n_{k}}\rightarrow{\infty}$ as ${k}\rightarrow{\infty}$, then $\nkp{0}$ is non-\ka; and likewise if $n_{k}$ divides $n_{k+1}$ for every $k\ge 0$.  Thus, for instance, $(2^{k})_{k\ge 0}$ is not a \ka\ sequence; but the rather ``close" sequence $(2^{k}+k)_{k\geq 0}$ turns out to be Kazhdan.
We refer to \cite{BG1, BG2, Ch, BaGrLyons} for more on \ka\ sequences.

\subsection{Rigidity sequences}\label{ss:1b}
According to Furstenberg and Weiss \cite{Fur}, a sequence of positive integers $(n_{k})_{k\ge 0}$ is said to be \emph{rigid} for a measure-preserving (dynamical) system $(X,\mathcal{B},m;T)$ on a probability space $(X,\mathcal{B},m)$, if ${m(T^{-n_{k}}A\, \Delta\, A)}\rightarrow{0}$ as ${k}\rightarrow{\infty}$ for every $A\in\mathcal{B}$. If we denote by $U_{T}$ the associated Koopman operator 
${f}\mapsto{f\circ T}$ on $L^{2}(X,\mathcal{B},m)$, this is equivalent to requiring that 
${||U_{T}^{n_{k}}f-f||}\rightarrow{0}$ as ${k}\rightarrow{\infty}$ for every $f\in L^{2}(X,\mathcal{B},m)$. A \emph{rigidity sequence} is a sequence $(n_{k})_{k\ge 0}$ which happens to be rigid for some weakly mixing dynamical system $(X,\mathcal{B},m\,;\,T)$. Recall that the \emph{weakly mixing} systems are those for which the spectral measure of the  operator $U_{T}$ acting on $L^{2}_{0}(X,\mathcal{B},m):=\{f\in L^{2}(X,\mathcal{B},m)\,;\,\int_{X}f\,dm=0\}$ is continuous. See for instance \cite{Wa} for more on this definition and on measurable dynamics in general.
\par\smallskip
Rigidity sequences were characterized in \cite{BDLR} and \cite{EG} in terms of Fourier coefficients of continuous measures on $\T$:
\begin{theorem}
\label{Th 2}
 A sequence of positive integers $(n_{k})_{k\ge 0}$ is a rigidity sequence if and only if there exists a measure $\mu \in\mathcal{P}_{c}(\T)$ such that ${\wh{\mu }(n_{k})}\rightarrow{1}$ as 
 ${k}\rightarrow{\infty.}$
\end{theorem}
Rigidity sequences are studied in detail in several papers, among which we mention   
\cite{BDLR, EG, A, FT, FK, Grie}. All examples of non-\ka\ sequences given above can be seen to be rigidity sequences. Adams proved in \cite{A}
that if there exists an irrational $z\in\T$ (\mbox{\it i.e.} $z$ is not a root of $1$) such that ${z^{n_{k}}}\rightarrow{1}$ as ${k}\rightarrow{\infty}$, then $(n_{k})_{k\ge 0}$ is a rigidity sequence; and a simpler proof of this result was found by Fayad and Thouvenot in \cite{FT}. The result was further generalized in \cite{BaGrLyons}, and this was applied to the resolution of a conjecture of Lyons \cite{L} related to Furstenberg's 
$\times 2\,$-$\times 3$ conjecture. On the other hand, examples of rigidity sequences $(n_{k})$ with the property that the set $\{ z^{n_k};\; k\geq 0\}$ is dense in $\T$ for every irrational $z\in\T$ 
were constructed in \cite{FK}.  Furthermore, Griesmer \cite{Grie} proved that there exist rigidity sequences $(n_{k})$ with the property that every translate $R$ of the set $\{n_k\,;\, k\ge 0\}$ is a \emph{set of recurrence} (in the terminology of \cite{FurstBook}, a \emph{Poincar\'e set}), which means that for any measure-preserving system $(X,\mathcal{B},m;T)$ and every $A\in\mathcal B$ with $m(A)>0$, there exists $r\in R\setminus\{ 0\}$ such that $m(A\cap T^{-r}A) > 0$. In particular, these rigidity sequences $(n_k)$ are dense in $\Z$ for the Bohr topology. However we note, paraphrasing \cite{Katz}, that a rigidity sequence cannot be uniformly distributed  ``in any reasonable sense". For example, it follows from Theorem \ref{Th 2} that a sequence $(n_k)_{k\geq 0}$ such that $(n_k\theta)$ is equidistributed modulo 1 for every $\theta \in\R\setminus\Q$ cannot be a rigidity sequence.

\par\smallskip
In view of the characterizations of Theorems \ref{Th 1} and \ref{Th 2}, it comes as a natural problem to investigate the links between  (non-)\ka\ sequences and rigidity sequences: are rigidity sequences non-\ka? And what about the converse? It was one of the initial motivations of this paper  to answer these questions (``Yes'' for the first one, ``No'' for the second one).

 \subsection{Nullpotent sequences} 
The notion of nullpotent sequence (the term was coined by Rusza \cite{R}), is defined as follows.
\begin{definition}\label{Definition 8}
 Let $\nkp{0}$ be a sequence of integers. We say that $(n_k)$ is \emph{nullpotent} if there exists a Hausdorff group topology $\tau $ on $\Z$ (i.e.\ a Hausdorff topology which turns 
 $\Z$ into a topological group) such that ${\nk}\rightarrow{0}$  for $\tau $ as ${k}\rightarrow{\infty}$.
\end{definition}

\par\smallskip
{Nullpotent sequences (in $\Z$ or in general abelian groups) have been studied, under different names, by several authors; see for instance {\cite{Graev,Nien1,Nien2,R,ZP}} as well as the recent survey \cite{dikran} and the references therein. Protasov and Zelenyuk \cite{ZP} and many subsequent authors (see \cite{dikran})
use the name \emph{$T$-sequences} instead of nullpotent sequences. As we apply several results from \cite{R}, we prefer to use Ruzsa's terminology. The notion of nullpotent sequence is directly related to that of rigidity sequence: indeed, it is easy to show that rigidity sequences are nullpotent (see Proposition \ref{Proposition 9} below). }

\smallskip
{The following  characterization of nullpotent sequences was obtained in \cite{R} and \cite{ZP}:
a sequence ${\nkp{0}}\subseteq\Z$ is  nullpotent if and only if, for any fixed $r\ge 1$, it is not possible to write any integer $n\neq 0$ as $n=\sum_{i=1}^{r}\varepsilon _{i}\,n_{k_{i}}$ with $\varepsilon _{i}=\pm 1$ and arbitrarily large indices $k_1,\dots ,k_r$.  }  

\smallskip
 Recall that a subset $D$ of $\Z$ is an \emph{additive basis} of $\Z$ if there exists some $r\in\N$ such that any integer $n\in\Z$ can be written as $n=\sum_{i=1}^r \varepsilon_i d_i$, where $d_i\in D$ and $\varepsilon_i=\pm 1$; and that $D$ is an \emph{asymptotic basis} of $\Z$ if this holds true for all but finitely many $n\in\Z$. In view of the above characterization of nullpotent sequences, it is not hard to convince oneself that if $(n_{k})_{k\ge 0}$ is a nullpotent sequence, then the associated set $\{n_{k}\,;\,k\ge 0\}$ cannot be an asymptotic basis of $\Z$. In fact, it is shown in \cite[Theorem 2]{R} that if $(n_k)$ is nullpotent then, for each fixed $r\in\N$, the set of all integers $n$ which can be written as
$n=\sum_{i=1}^{r}\varepsilon _{i}\,n_{k_{i}}$ with $\varepsilon _{i}=\pm 1$
is of density zero in $\Z$. On the other hand, it is not hard to check (for example by using Theorem \ref{Th 1} and the estimates in Fact \ref{Fact 0} below) that asymptotic bases of $\Z$ are Kazhdan sets. This applies in particular to the set of all squares, by Lagrange's four-square theorem, or to the set of all primes, by Vinogradov's theorem. Therefore,  
 it  comes as a natural question to wonder whether there exist \ka\ 
sequences  of integers which do not form (asymptotic) bases of $\Z$. This was asked by Martin Kassabov to the first-named author in private communication. We answer this question affirmatively in Section \ref{Section 2} of the paper (Corollary \ref{cor:base}), by considering random sequences of integers. More precisely, we show that there exist Kazhdan sets which are even nullpotent.

\subsection{Structure of the paper} The remaining of the paper is organized as follows. Section \ref{prelim} contains a few preliminary facts.  
In Section~\ref{Section 2}, we study the mutual implications and non-implications between the properties we are considering. We show that rigidity sequences are both non-\ka\ and nullpotent (Corollary \ref{Theorem 3} and Corollary 
\ref{rigid/nullpotent}). We also show that a sequence which converges to zero for some precompact group topology on $\Z$ is a rigidity sequence (Theorem~\ref{TBrigid}); the converse is false. On the other hand, we observe that there there exist non-\ka\ sequences which are not nullpotent, and hence non-rigid (Proposition \ref{Proposition 4}); and we show that in a suitable probabilistic setting, almost all sequences are both Kazhdan and nullpotent (Theorem \ref{Theorem 10}). In Section~\ref{Section 3}, we use Baire category methods to obtain  a useful characterization of rigidity sequences (Theorem \ref{Theorem 5}), which allows us to retrieve all examples from 
\cite{BDLR,EG,A,FT,BaGrLyons}; and we strengthen this further in Theorem \ref{vrai?}. 
These results shed in particular an interesting light on the constructions of \cite{BaGrLyons} and \cite{FK}.
In Section~\ref{sect:bZ}, we exhibit an explicit example of a rigidity sequence which is dense in $\Z$ for the Bohr topology; and we also make some comments on Griesmer's original proof. We end by stating a few open problems.

\section{Preliminary facts}\label{prelim}
\subsection{Easy estimates} The following estimates on quantities of the form 
$|\much(n)-1|$ will be used repeatedly in the sequel, sometimes without explicit mention.
\begin{fact}\label{Fact 0}
 Let $\mu \in \mpt$. For any $m,n\in\Z$, we have
 \begin{enumerate}
\item[{(1)}] $|\much(n)-1|\le\ds\int_{\T}|z ^{n}-1|\,d\mu (z )\le\sqrt{2}\,\,|\much(n)-1|^{1/2}$;
\item[(2)] $|\wh{\mu}(m+n)-1|\le\sqrt{2}\,\Bigl(|\much(m)-1|^{1/2}+|\wh{\mu }(n)-1|^{1/2}\Bigr).$
\end{enumerate}
\end{fact}

\begin{proof}
The proof of (1) runs as follows:
\begin{align*}
|\much(n)-1|&\le\int_{\T}|z ^{n}-1|\,d\mu (z )\le\Bigl(\int_{\T}|z ^{n}-1|^{2}d\mu (z )\Bigr)^{1/2}\\
&=\Bigl(2\Re e\int_{\T}(1-z ^{n})\,d\mu (z )\Bigr)^{1/2}\le\sqrt{2}\,\Bigl|\int_{\T}(1-z ^{n})\,d\mu (z )\Bigr|^{1/2}\\
&=\sqrt{2}\,|\much(n)-1|^{1/2}.
\end{align*}
Part (2) follows from (1) since $\vert z^{m+n}-1\vert\leq \vert z^{m+n}-z^n\vert +\vert z^n-1\vert=\vert z^m-1\vert +\vert z^n-1\vert$ for every $z\in\T$.
\end{proof}

As a direct consequence of Fact \ref{Fact 0}, we get

\begin{fact}\label{Fact 00}
Let $(\mk)_{k\ge 0}$ and $\nkp{0}$ be two sequences of integers, and let $\varepsilon >0$. 
\par\smallskip
\begin{enumerate}
\item[{(1a)}] ${\much(n_{k})}\rightarrow{1}$ as ${k}\rightarrow{\infty}$ if and only if 
${\ds\int_{\T}|z ^{n_{k}}-1|\,d\mu (z )}\rightarrow{1}$; 
\par\smallskip
\item[{(1b)}] if $\sup\limits_{k\ge 0}|\much(n)-1|<\varepsilon $ then  
$\ds\sup_{k\ge 0}\int_{\T}|z ^{n_{k}}-1|\,d\mu (z )<\sqrt{2\varepsilon }$.
\par\medskip
\item[{(2a)}] If ${{\wh{\mu }(\mk)}\rightarrow{1}}$ and ${\wh{\mu }(\nk)}\rightarrow{1}$, then ${\wh{\mu }(\mk\pm\nk)}\rightarrow{1;}$
\par\medskip
\item[{(2b)}] if $\ds\sup_{k\ge 0}|\wh{\mu }(\nk)-1|<\varepsilon $ and $\ds\sup_{l\ge 0}|\wh{\mu }(m_l)-1|<\varepsilon $, then 
$
\ds\sup_{k,l\ge 0}|\wh{\mu }(\nk\pm m_l)-1|<2\,\sqrt{2\varepsilon }.
$
\end{enumerate}
\end{fact}

\subsection{The set of rigid measures} Let $(n_k)_{k\geq 0}$ be a sequence of integers. A measure $\mu\in\mathcal P(\T)$ is said to be \emph{rigid} for $(n_k)$ if $\widehat\mu(n_k)\to 1$ as $k\to\infty$. We denote by $\rnk$ the set of all such measures: 

\[ \mathcal R^{(n_k)}:=\bigl\{ \mu\in\mathcal P(\T)\,;\,\widehat\mu(n_k)\to 1\quad\hbox{as $k\to\infty$}\bigr\}.\]

\smallskip\noindent
This set has several interesting stability properties, which we summarize in the next lemma. These properties are well known, and can be found for example in \cite[Chapter 7]{Nad}; but we include the proofs for convenience of the reader. 

\smallskip
The following notations will be used in this section and elsewhere in the paper. If $\mu\in\mathcal P(\T)$, we denote by $L^1(\mu)$ the set of all measures absolutely continuous with respect to $\mu$. Also, denote by $G(\mu)$ the subgroup of $\T$ generated by the support of $\mu$. 
For any $\mu\in\mathcal P(\T)$ and any $p\in\Z$, denote by $\mu_{{\#} p}$ the image measure of $\mu$ under the map $z\mapsto z^p$ from $\T$ into itself. (In particular, $\mu_{\#0}=\delta_1$.) Lastly, write $e_n(z):=z^n$ for every $n\in \Z$ and every $z\in\T$.

\begin{lemma}\label{Fact 0 bis}
 Let $\nkp{0}$ be a sequence of integers.
 \begin{itemize}
 \item[\rm (1)] The set $\mathcal R^{(n_k)}$ is convex, contains the Dirac mass $\delta_1$, and is closed under convolution.
 \item[\rm (2)]  For any $p\in\Z$, the set $\mathcal R^{(n_k)}$ is closed under the map $\mu\mapsto\mu_{\# p}$. 
 \item[\rm (3)] The set $\mathcal R^{(n_k)}$ is hereditary for absolute continuity: if $\mu\in\mathcal R^{(n_k)}$ and if $\nu \in \mpt$ is absolutely continuous with respect to $\mu $, then $\nu\in\mathcal R^{(n_k)}$.
\end{itemize}
 \end{lemma}
\begin{proof} (1) This is obvious.

(2) Observe that $\widehat{\,\mu_{\# p}}(n)=\widehat\mu(pn)$ for every $n\in\Z$. This gives
\[ \vert \widehat{\,\mu_{\# p}}(n)-1\vert\leq \int_\T \vert z^{pn}-1\vert\, d\mu\leq \vert p\vert\int_\T \vert z^n-1\vert\, d\mu,\]
so the result follows from Fact \ref{Fact 00}. 

(3) By Fact \ref{Fact 00}, $e_{n_k}\to 1$ in the $L^1(\mu)$ norm, and hence in $\mu$-measure. 
Since $\nu$ is absolutely continuous with respect to $\mu$, it follows that $e_{n_k}\to 1$ in $\nu$-measure, and hence in the $L^1(\nu)$ norm because the sequence $(e_{n_k})$ is uniformly bounded. So $\widehat\nu(n_k)\to 1$, by Fact \ref{Fact 00} again.
 \end{proof}

\begin{remark}\label{quivadesoi}
If $\nu$ is a \emph{complex} measure with $\Vert \nu\Vert=1$ which is absolutely continuous with respect to $\mu\in\mathcal R^{(n_k)}$, then property (3) of Lemma \ref{Fact 0 bis} applied to $\vert \nu\vert$ yields that $\widehat\nu(n_k)\to \widehat\nu(0)$, which is equal to $1$ only if $\nu$ is a probability measure. If we take for example $d\nu =\bar z^q\, d\mu$ for some $q\in\N$, this will hold if and only if $\bar z^q=1$ $\mu$-{a.e.}, \mbox{\it i.e.} $\mu$ is supported on the $q$-roots of $1$. In particular, if $\mu$ is a continuous measure then $\widehat\mu(n_k)$ and $\widehat\mu(n_k+q)$ cannot tend to $1$ simultaneously. 
In the same spirit, the following simple fact may be worth mentioning. Let $\nkp{0}$ be a sequence of integers, and let $\mu\in\rnk$. Let also $(p_k)$ and $(q_k)$ be two arbitrary subsequences of $(n_k)$. Then $\mu\bigl(\{ z\in\T;\; z^{q_k-p_k} \;\hbox{has a limit $\phi(z)\neq 1$}\}\bigr)=0$. This is indeed clear since $\vert z^{q_k-p_k}-1\vert=\vert z^{q_k}-z^{p_k}\vert\to 0$ in $L^1(\mu)$.
 \end{remark}
 
In the next two corollaries, we point out some useful consequences of Lemma \ref{Fact 0 bis}. 
 \begin{corollary}\label{density} For any $\mu\in\mathcal R^{(n_k)}$, the closure of $\mathcal R^{(n_k)}\cap L^1(\mu)$ in $\mathcal P(\T)$ contains every measure supported on $\overline{G(\mu)}$. In particular, if $(n_k)$ is a rigidity sequence, then $\mathcal R^{(n_k)}\cap\mathcal P_c(\T)$ is dense in $\mathcal P(\T)$.
 \end{corollary}
 \begin{proof} Fix $\mu\in\mathcal R^{(n_k)}$, and set $\mathcal H:={\mathcal R^{(n_k)}\cap L^1(\mu)}$. By Lemma \ref{Fact 0 bis}, $\mathcal H$ is a convex subset of $\mathcal P(\T)$ closed under convolution and hereditary for absolute continuity. If $z\in {\rm supp}(\mu)$, then the Dirac mass $\delta_z$ can be approximated by measures absolutely continuous with respect to $\mu$. Hence, $\overline{\mathcal H}$ contains $\delta_z$ for every $z\in {\rm supp}(\mu)$. Now, $\overline{\mathcal H}$ is closed under convolution, because $\mathcal H$ is closed under convolution and the convolution product is separately continuous. So $\overline{\mathcal H}$ contains $\delta_z$ for every $z\in G(\mu)$. Since $\overline{\mathcal H}$ is also convex, it contains every measure $\nu\in\mathcal P(\T)$ whose support is finite and contained in ${G(\mu)}$; and since $\overline{\mathcal H}$ is closed, it contains in fact every measure supported on $\overline{G(\mu)}$.
 \par\smallskip
 If now $(n_k)$ is a rigidity sequence, let $\mu$ be a continuous measure belonging to $\mathcal R^{(n_k)}$. The support of $\mu$ is uncountable, so it contains an irrational $z$, 
 and hence $\overline{G(\mu)}=\T$. Since $\mathcal R^{(n_k)}\cap L^1(\mu)$ is contained in $\mathcal R^{(n_k)}\cap \mathcal P_c(\T)$, this shows that the latter set is dense in $\mathcal P(\T)$. 
  \end{proof}

  \begin{corollary}\label{densitybis} Set $E^{(n_k)}:=\bigcup\bigl\{ G(\mu);\; \mu\in\rnk\bigr\}$, and let us denote by $G^{(n_k)}$ the subgroup of $\T$ generated by $E^{(n_k)}$.  If $G^{(n_k)}$  is dense in $\T$, then $\rnk$ is dense in $\mathcal P(\T)$.
  \end{corollary}
  \begin{proof} By Corollary \ref{density}, the closure $\overline{\rnk}$ of $\rnk$ in $\mathcal P(\T)$ contains every measure $\nu$ whose support is contained in $G(\mu)$ for some $\mu\in\rnk$. In particular, $\delta_z\in \overline{\rnk}$ for every $z\in E^{(n_k)}$. 
  By convolution-invariance and convexity, it follows that $\overline{\rnk}$ contains every measure $\nu$ whose support is finite and contained in $G^{(n_k)}$, and this concludes the proof.
  \end{proof}
  
  \begin{remark}The  converse statement is obviously true: if $\rnk$ is dense in $\mathcal P(\T)$, then $E^{(n_k)}$ is infinite and hence the group $G^{(n_k)}$ is dense in $\T$.
  \end{remark}

\section{Implications and non-implications}\label{Section 2}

In this section, we show that rigidity sequences are both non-Kazhdan and nullpotent, but that as far as only these classes of sequences are considered, the converse implications fail in the strongest possible senses, \mbox{\it i.e.} there exist non-Kazhdan sequences which are not nullpotent, and nullpotent sequences which are Kazhdan. In particular, there exist Kazhdan sequences which are not asymptotic bases of $\Z$.
  
  \subsection{Rigid implies non-Kazhdan and nullpotent}\label{ss:3a} Our first aim is to prove that rigidity sequences are non-Kazhdan.
We will in fact prove the following slightly stronger result.

\begin{theorem}\label{Mazur} Let $(n_k)_{k\geq 0}$ be a rigidity sequence. Then, for any measure $\nu\in\mathcal R^{(n_k)}$ and any $\varepsilon >0$, one can find a measure $\mu\in \mathcal R^{(n_k)}\cap \mathcal P_c(\T)$ such that $$\sup_{k\geq 0} \vert \widehat\mu(n_k)-\widehat\nu(n_k)\vert<\varepsilon.$$
\end{theorem}

\begin{proof}Recall first the classical  Mazur's theorem (see \mbox{e.g.} \cite[page 216]{Meg}): in a normed space, the weak closure of any convex set is equal to its norm closure.  Recall also that $\ell_\infty(\Z_+)$ is the dual space of $\ell_1(\Z_+)$, which is itself the dual space of $c_0(\Z_+)$; so we can consider the $w^*$ topology on $\ell_\infty(\Z_+)$, and the topology induced on $c_0(\Z_+)$ is the weak topology of $c_0(\Z_+)$.

Consider the map $T:\mathcal P(\T)\to \ell_\infty(\Z_+)$ defined by 
\[ T\mu:=\bigl( \widehat\mu(n_k)-1\bigr)_{k\geq 0}.\]
Since $T\mu-T\nu=(\widehat\mu(n_k)-\widehat\nu(n_k))_{k\geq 0}$, what we have to show is the following: if $\nu\in\mathcal R^{(n_k)}$, then 
\[ T\nu\in \overline{T\bigl(\mathcal R^{(n_k)}\cap \mathcal P_c(\T)\bigr)}^{\Vert \;\cdot\;\Vert_\infty}.\]
Now, the map $T$ is affine, and continuous from $\mathcal P(\T)$ into $\ell_\infty(\Z_+)$ endowed with its $w^*$ topology. By Corollary \ref{density}, it follows that the set $\mathcal C:=T\bigl(\mathcal R^{(n_k)}\cap \mathcal P_c(\T)\bigr)$ is convex and $w^*$-dense in $T(\mathcal P(\T))$. But $T(\mathcal R^{(n_k)})$ is contained in $c_0(\Z_+)$ by the definition of $T$. So $T\bigl(\mathcal R^{(n_k)}\cap \mathcal P_c(\T)\bigr)$ is convex and \emph{weakly} dense in $T(\mathcal R^{(n_k)})\subseteq c_0(\Z_+)$, and hence \emph{norm}-dense in $T(\mathcal R^{(n_k)})$ by Mazur's theorem. This concludes the proof of Theorem \ref{Mazur}.
\end{proof}

From Theorem \ref{Mazur}, we immediately deduce 

  \begin{corollary}\label{Theorem 3} Rigidity sequences are non-Kazhdan.
\end{corollary}

\begin{proof} By Theorem \ref{Th 1}, it is enough to show that if $\nkp{0}$ is a rigidity sequence then, for any $\varepsilon >0$, one can find a measure $\mu\in\mathcal P_c(\T)$ such that
$\sup_{k\ge 0}|\widehat\mu (\nk)-1|<\varepsilon .$   This follows from Theorem~\ref{Mazur} by taking $\nu:=\delta_1$.
  \end{proof}
 
  \begin{remark}\label{Remark 3 bis}
  From the proof of Theorem \ref{Theorem 3}, one can obtain a more precise statement. Let $(\nk)_{k\ge 0}$ be a rigidity sequence, and choose a continuous measure $\nu\in\mathcal P(\T)$ such that $\widehat\nu(n_k)\to 1$. Recall that for any $p\in\Z$, we denote by $\nu_{\# p}$ the image of $\nu$ under the map $z\mapsto z^p$ (so that $\nu_0=\delta_1$ and $\nu_{\# p}\in \mathcal R^{(n_k)}\cap \mathcal P_c(\T)$ for all $p\neq 0$ by Lemma \ref{Fact 0 bis}).  Then $\nu_{\# n_k}\to \delta_1$ as $k\to\infty$, because $\widehat{\,\nu_{\# n_k}}(n)=\widehat\nu(n_kn)=\widehat{\nu_{\# n}}(n_k)\to 1$ for every $n\in\Z$. Applying Mazur's theorem as above, we conclude that for any $\varepsilon$, one can find a continuous measure $\mu$ which is a \emph{convex combination of the measures $\nu_{\# n_k}$} such that $\sup_{k\geq 0}\vert \widehat\mu(n_k)-1\vert<\varepsilon$. \end{remark}
  
 \begin{remark}\label{rem:egg} From Theorem \ref{Theorem 3} and results from \cite{EG,BDLR}, it becomes clear that if $n_{k+1}/n_{k}\to \infty$, or if $n_k$ divides $n_{k+1}$ for all $k$ then $\nkp{0}$ is a non-Kazhdan sequence. Indeed, it is shown for example in \cite[Example 3.4 and Proposition 3.9]{EG} that these assumptions imply rigidity. This can also be deduced from \cite[Theorem 2.3]{BaGrLyons}, using a classical result of Eggleston \cite{omelette} in the case where $n_{k+1}/n_k\to\infty$.  \end{remark}

\smallskip Now we show that rigidity sequences are also nullpotent. This will follow at once from the next result.
\begin{proposition}\label{Proposition 9}
 Let $\nkp{0}$ be a sequence of integers. Assume that for every integer $q\geq 1$, there exists a measure $\mu_q\in\rnk$ such that $\mu_q\bigl(\{ z\in\T;\; z^q=1\}\bigr)<1$.  
Then $(n_k)$ is a nullpotent sequence.
\end{proposition}

\begin{proof}
 Define a distance on $\Z$ by setting
 \[
d(m,n):=\sum_{q=1}^\infty 2^{-q} \int_{\T}|z ^{m}-z ^{n}|\,d\mu_q (z )\quad \textrm{for every}\ m,n\in\Z.
\]
This is indeed a distance on $\Z$: if $d(m,n)=0$, then $z ^{m-n}=1$ $\sigma_q $-almost everywhere for every $q\geq 1$, hence $m=n$ by the choice of the measures $\mu_q$. This distance $d$ is translation-invariant, so it defines a group topology on $\Z$; and $d(\nk,0)=\sum\limits_{q=1}^\infty 2^{-q}\ds\int_{\T}|z ^{\nk}-1|\,d\mu_q (z )\rightarrow{0}$ as 
${k}\rightarrow{\infty}$ because $\mu_q\in\rnk$ for all $q\geq 1$.
\end{proof}

\smallskip
\begin{corollary}\label{rigid/nullpotent} Rigidity sequences are nullpotent.
\end{corollary}
The converse implication is false: there exist sequences $(n_k)$ which tend to zero in a suitable Hausdorff group topology $\tau$ of $\Z$ but which are not rigidity sequences. This will be proved in Theorem \ref{Theorem 10}. One may ask what additional properties the topology $\tau$ should satisfy in order to force the rigidity of the sequence. The following definition has been introduced in \cite{DD}. Recall that a Hausdorff topological group $(G,+)$ is said to be 
\emph{precompact} if, for every non-empty open set $U$ of $G$, there exists a finite set $F\subseteq G$ such that $G=U+F$.

\begin{definition}
Let $\nkp{0}$ be a sequence of integers. We say that $(n_k)$ is a \emph{totally bounded} sequence, or a \emph{$TB$-sequence} for short, if there exists a {precompact} group topology $\tau $ on $\Z$ 
such that ${\nk}\rightarrow{0}$  for $\tau $ as ${k}\rightarrow{\infty}$.
\end{definition}

\begin{theorem}\label{TBrigid}
$TB$-sequences are rigidity sequences.
\end{theorem}
\begin{proof}
Let $\nkp{0}$ be a sequence of integers. It is proved in \cite[Proposition 2.4]{DD} that $(n_k)$ is a $TB$-sequence if and only if the subgroup 
$$\Gamma^{(n_k)}:=\{z \in\T\,;\, {z ^{\nk}}\rightarrow{1}\}$$ 
is infinite. Moreover, it is proved in \cite[Theorem 2.3]{BaGrLyons} that if $\Gamma^{(n_k)}$ is dense in $\T$, then $(n_k)$ is a rigidity sequence. Since a subgroup of $\T$ is dense if and only if it is infinite, the result follows immediately.
\end{proof}
The converse of Theorem~\ref{TBrigid} is false: as mentioned in the introduction Fayad and Kanigow\-ski constructed in \cite{FK} examples of rigidity sequences $\nkp{0}$ for which $\Gamma^{(n_k)}=\{ 1\}$. This will be discussed again in Subsection~\ref{ss:4a}, after Corollary~\ref{points}.

\begin{remark}
It follows from Remark~\ref{rem:egg} that sequences $\nkp{0}$ with $n_{k+1}/n_k\to\infty$ are $TB$-sequences; see also \cite[Theorem 3.1]{DD}. 
\end{remark}

\begin{remark} According to a well established terminology, the subgroup $\Gamma^{(n_k)}$ considered in the proof of Theorem \ref{TBrigid} is called the subgroup of $\T$ \emph{characterized by the sequence $(n_k)$}. There is a considerable literature on characterized subgroups; see \cite{dikran}.
\end{remark}

\subsection{Non-Kazhdan does not imply nullpotent} The following example shows that non-Kazhdan sequences may fail to be nullpotent.

\begin{proposition}\label{Proposition 4}
 Let $(m_{k})_{k\ge 0}$ be a sequence of integers with $m_{0}=1$, 
 and assume that $(m_{k})$ is non-\ka. Write the set $\{m_{k},m_{k}+1\,;\,k\ge 0\}$ as a strictly increasing sequence $\nkp{0}$. Then, $(n_k)$ is a non-\ka\ sequence  which is not nullpotent (and hence not a rigidity sequence).
 \end{proposition}
\begin{proof} This will follow at once from the next two facts.

\begin{fact}\label{truc1} Let $(n_k)_{k\geq 0}$ be a sequence of integers. If $(n_k)$ is nullpotent then, for any $c_0,\dots ,c_p\in\Z$, the only accumulation points of 
$\vert c_0n_k+c_1n_{k+1}+\cdots +c_p n_{k+p}\vert$ in $\Z_+\cup\{\infty\}$ are $0$ and $\infty$; in other words, if the set $\mathbf K:=\{ k\geq 0;\;  c_0n_k+c_1n_{k+1}+\cdots +c_p n_{k+p}\neq 0\}$ is infinite, then $\vert c_0n_k+c_1n_{k+1}+\cdots +c_p n_{k+p}\vert\to\infty$ as $k\to\infty$, $k\in\mathbf K$. In particular, if $(n_k)$ is nullpotent and if the $n_k$ are all distinct, then  $\vert n_{k+1}-n_k\vert\to\infty$.
\end{fact}
\begin{proof}[Proof of Fact \ref{truc1}] Assume that the conclusion fails. Then there exists an integer $q\neq 0$ such that  $c_0 n_k+ c_1n_{k+1}+\cdots +c_p n_{k+p}=q$ for infinitely many $k$; and this immediately implies that $(n_k)$ cannot be nullpotent.
\end{proof}

\begin{fact}\label{truc2} Let $Q\subseteq\Z$
, and assume that $Q$ is not a Kazhdan set. Then the set $Q\cup (Q+Q)$ is non-Kazhdan. In particular, $Q\cup(Q+a)$ is non-Kazhdan for any $a\in Q$.
\end{fact}
\begin{proof}[Proof of Fact \ref{truc2}] Let $\varepsilon >0$ be arbitrary. Since $Q$ is 
non-Kazhdan, one can find a 
measure $\mu\in\mathcal P(\T)$ such that $\mu(\{ 1\})=0$ and $\sup_{q\in Q} \vert\widehat\mu(q)-1\vert<\varepsilon$. Then, by (2) of Fact \ref{Fact 00}, we have $\sup_{r\in Q+Q} \vert\widehat\mu(r)-1\vert< 2\sqrt{2}\varepsilon $. This being true for every $\varepsilon >0$, it follows that $Q\cup (Q+Q)$ is non-Kazhdan.
\end{proof}

Since $m_0=1$, if we set $Q:=\{ m_k;\; k\geq 0\}$, then $Q\cup (Q+1)$ is non-Kazhdan by Fact \ref{truc2}; so the sequence $(n_k)$ is non-Kazhdan. That $(n_k)$ cannot be a rigidity sequence follows immediately from Fact \ref{truc1}.
\end{proof}

\begin{remark} Fact \ref{truc1} shows in particular that if $\nkp{0}$ is a strictly increasing rigidity sequence, then $n_{k+1}-n_k\to\infty$. A more general result was proved in \cite[Proposition 2.12]{BDLR}. Moreover, the proof given in \cite[Example 6.9]{BaGrLyons} shows that if $(m_k)_{k\ge 0}$ is a strictly
 increasing sequence of positive integers with ${m_{k+1}-m_k}\rightarrow{\infty}$, then there exists a strictly increasing sequence $(n_{k})_{k\ge 0}$ which is a $TB$-sequence (hence a rigidity sequence by Theorem \ref{TBrigid}) and satisfies $m_k\le n_k < m_{k+1}$ for every $k\ge 0$. 
 \end{remark}

\begin{remark} In the situation described in Proposition \ref{Proposition 4}, it is not hard to see that the set
\[
\mathcal{R}_{\varepsilon }^{(\nk)}:=\Bigl\{\mu\in\mathcal{P}(\T)\,:\,\;\overline{\lim_{k\to\infty}}\;
|\much(\nk)-1|<\varepsilon\Bigr\}
\]
is not dense in $\mathcal{P}(\T)$ for $\varepsilon >0$ sufficiently small. 
Indeed, any measure $\mu \in\rnk_\varepsilon$ satisfies 
\[
|\much(1)-1|\le \overline{\lim_{k\to\infty}}\;
\sqrt{2}\bigl(|\much(m_{k}+1)-1|^{1/2}+|\much(m_{k})-1|^{1/2}\bigr)<2\sqrt{2}\varepsilon ,
\]
which makes density impossible for all $\varepsilon <1/(2\sqrt{2})$ (the Lebesgue measure on $\T$ cannot be approximated by measures from $\mathcal{R}_{\varepsilon }^{(\nk)}$). We will see in Section \ref{Section 3} that this is not accidental.
\end{remark}

\subsection{A Kazhdan sequence is almost surely nullpotent} We now show that the converse of Corollary~\ref{rigid/nullpotent} is not true: there exist nullpotent sequences which are not rigidity sequences. In fact, we are going to prove that there exist nullpotent sequences which are even Kazhdan. 
In order to do this, we need, as in \cite{R}, to consider \emph{random sequences} of integers. 

\smallskip Let $\mathbf p=(p _{n})_{n\ge 1}$ be a sequence of real numbers with $0<p_{n}<1$ for every $n\ge 1$. The random set $A=A^{\mathbf p}$ associated with $\mathbf p$ is defined by putting a given integer $n$ into $A$ with probability $p_n$ in such a way that the events $\{ n\in A\}$ are independent. More formally, let  $(\xi _{n})_{n\ge 1}$ be a sequence of independent random variables on  a probability space 
$(\Omega ,\mathcal{F},\P)$ with
\[
\P(\xi _{n}=1)=p_{n}\quad \textrm{and}\quad \P(\xi _{n}=0)=1-p_{n}\quad \textrm{for every}\ n\ge 1.
\]
Then, for any $\omega\in\Omega$, 
\[
n\in A(\omega )\iff \xi_n(\omega)=1.
\]
If the sequence $\mathbf p$ satisfies $\sum_{n=1}^{\infty}p_{n}=\infty$, then $A(\omega )$ is almost surely infinite by the Borel-Cantelli Lemma. So we may enumerate $A(\omega)$ as a strictly increasing sequence $(n_k(\omega))_{k\geq 0}$. Equivalently, $n_k(\omega)$ is the smallest $r\in\N$ such that $\xi_{1}(\omega)+ \cdots +\xi_{r}(\omega) = k$. We say that $(n_k)_{k\geq 0}$ is \emph{the random sequence associated with $\mathbf p$}. 

 \begin{theorem}\label{Theorem 10}
 Assume that the sequence $\mathbf p=(p_n)_{n\geq 1}$ satisfies the following three conditions:
  \begin{enumerate}
\item[\emph{(1)}] $\dfrac{1}{\log N}\,\ds\sum_{n=1}^{N}p_{n}\rightarrow{\infty}$ as 
${N}\rightarrow{\infty}$;
\item[\emph{(2)}] $\ds\sum_{n=1}^N \vert p_{n+1}-p_n\vert=\emph{o} \Bigl(\sum_{n=1}^{N}p_{n}\Bigr)$ as $N\to\infty$;
\item[\emph{(3)}] $p_{n}=\emph{O}\Bigl(\dfrac{1}{n^{1-\varepsilon }}\Bigr)$ as $n \rightarrow \infty$ for every $\varepsilon \in(0,1)$.
\end{enumerate}
Then, almost surely, the random sequence $(n_k)$ is nullpotent, Kazhdan (so not rigid), and such that every translate of the set $\{ n_k;\; k\geq 0\}$ is a set of recurrence. 
 \end{theorem}
 
 \smallskip
 \begin{remark} Condition (2) holds true as soon as the sequence $(p_n)_{n\ge 1}$ is decreasing and $\sum_{n=1}^\infty p_n=\infty$ (which follows from (1)). To give a concrete example, the three assumptions on $\mathbf p$ are satisfied if we take $p_n:=\frac{\log(n)}{n}$ for $n\geq 2$. 
 \end{remark}

\smallskip  

The proof of Theorems \ref{Theorem 10} will be merely a combination of known results. The first one was proved by Ruzsa in \cite[Theorem 4]{R}.

\begin{theorem}\label{utile1} If $\mathbf p$ satisfies \emph{(3)} and $\sum_{n=1}^\infty p_n=\infty$, then the random set $A$ is almost surely nullpotent.
\end{theorem}

The second result we need goes back to Bourgain \cite{B}. Recall that a sequence $(n_k)_{k\geq 0}$ of elements of $\Z$ is said to be \emph{Hartman equidistributed} (see \cite[page 295]{KuiNie}) if 
\[ \frac1{K+1}\sum_{k=0}^K z^{n_k}\to 0\qquad\hbox{for every $z\in\T\setminus\{ 1\}$.}\]

\begin{theorem}\label{utile2} If $\mathbf p$ satisfies \emph{(1)} and \emph{(2)} then the random sequence $(n_k)$ is almost surely Hartman equidistributed.
\end{theorem}

This was proved by Bourgain in \cite[Proposition 8.2]{B} under slightly stronger assumptions. A more general result is obtained in \cite[Theorem 5.4]{N}, with a quite different proof from that in \cite{B}. As stated, the result can be found in \cite[Theorem 6.1]{China}. We mention also the references \cite{Bos83,LQR,FLWsz,Fr2012} for more results in this vein. 

\par\smallskip
The last result we need reads as follows.
\begin{theorem}\label{utile3} Let $(n_k)_{k\geq 0}$ be a sequence of integers. Assume that  $(n_k)$ is Hartman equidistributed.  Then the set $\{\nk\,;\,k\ge 0\}$ is \ka, and every translate of $\{ n_k;\; k\geq 0\}$ is a set of recurrence.
\end{theorem}
\begin{proof} The equidistribution assumption implies that 
\begin{equation}\label{Hartman} 
 \frac{1}{K+1} \sum_{k=0}^K \widehat\mu(n_k)\to \mu(\{ 1\})
 \end{equation}
 for every  finite positive measure $\mu$ on $\T$. 
So the first statement is a simple consequence of Theorem \ref{Th 1}. The second statement is well known (see \mbox{e.g.} \cite[Proposition 3.4]{Bos83} and the proof of \cite[Theorem 3.5]{Furst}), and can be proved by considering the spectral measures of the Koopman operator $f\mapsto f\circ T$ associated to a measure-preserving transformation 
$T:(X,\mathcal{B},m)\to (X,\mathcal{B},m)$. More precisely, given a set $A\subseteq X$ with $m(A)>0$, suppose that $m\bigl(T^{-n_k}(A)\cap A\bigr)=0$ for every $k\geq 0$. Let $\mu$ be the positive measure such that $\widehat\mu(n)=\langle \mathbf 1_A\circ T^n,\mathbf 1_A\rangle_{L^2(m)}$, $n\geq 0$.
An application of (\ref{Hartman}) to $\mu$ yields that $\mu(\{ 1\})=0$.
By von Neumann's mean ergodic theorem and since $\mu(\{1\})=\lim_{N\to\infty}\frac1{N+1} \sum_{n=0}^N \widehat\mu(n)$, it follows that $\mathbf 1_A$ is orthogonal to its projection onto the subspace of all $T$-invariant functions in $L^2(m)$. So $\mathbf 1_A$ is in fact orthogonal to all $T$-invariant functions; and hence $m(A)=\pss{\mathbf 1_A}{\mathbf 1}=0$, a contradiction. 
Thus, we see that $\{ n_k;\; k\geq 0\}$ is a set of recurrence; and the same is true for any translate of $\{ n_k;\; k\geq 0\}$ because the sequence $(n_k+m)$ is Hartman equidistributed for any $m\in\Z$.
\end{proof}

\begin{proof}[Proof of Theorem \ref{Theorem 10}] The result follows immediately from Theorems \ref{utile1}, \ref{utile2} and \ref{utile3} above.
\end{proof}

\subsection{Kazhdan sequences which are not asymptotic bases of $\Z$}
Although the most explicit and natural examples of nullpotent sequences enjoy the stronger property of being rigidity sequences (this is typically the case for sequences $\nkp{0}$ such that ${n_{k+1}/n_{k}}\rightarrow{\infty),}$ Theorem \ref{Theorem 10} shows that from a probabilistic point of view,  a sequence of integers is rather inclined to be nullpotent but not rigid. 
In the same way, most known examples of \ka\ sets in $\Z$ (such as the set of squares, or the set of all primes...) are additive bases of $\Z$, but a set of integers taken at random will rather be  \ka\ and {not} an asymptotic basis. As mentioned in the introduction, this answers a question of Martin Kassabov. However, we must also add that presently we do not know any explicit example of a Kazhdan set which is not an additive basis.

\begin{corollary}\label{cor:base} Under the assumptions of Theorem \ref{Theorem 10} above,
the random sequence $(n_k)$ almost surely satisfies the following property: 
$(n_k)$ is a \ka\ sequence and, for each fixed integer $r\ge 1$, the set $B_r$ of all integers $m$ which can be written as 
\[
m=\sum_{i=1}^{r}\varepsilon _{i}\,n_{k_{i}},\qquad \varepsilon _{i}=\pm 1
\]
is of density zero in $\Z$. In particular, $(n_k)$ is almost surely a Kazhdan sequence and not an asymptotic basis of $\Z$.
\end{corollary}
\begin{proof}
It follows from Theorem \ref{Theorem 10} that  the sequence $(n_k(\omega))_{k\ge0}$ is almost surely Kazhdan and nullpotent; and it is proved in \cite[Theorem 2]{R} that all the sets $B_r$ associated to a nullpotent sequence are of density zero in $\Z$.
\end{proof}
\begin{remark} The choice $p_n=\frac{\log(n)}{n}$ for $n\geq 2$ in Theorem \ref{Theorem 10} gives rise to random sequences $(n_k)$ which are almost surely nullpotent.  On the other hand, it has been remarked by Ruzsa \cite{R} that if we choose $p_n = n^{-d}$ for some $0<d<1$, then the random sequence $(n_k)$ is almost surely an asymptotic basis of $\Z$ and hence not nullpotent. 
Additive properties of random sequence of (positive) integers have been studied by many authors, beginning with Erd\"{o}s and Renyi \cite{ErdosRenyi}. 
\end{remark}

\begin{remark}
One can also wonder what happens from a \emph{topological} point of view. As it turns out, the situation is quite different: identifying a subset of $\N$ with a point of the Cantor space $\{ 0,1\}^\N$, 
it is very easy to show that the set of all additive bases of order $2$ of $\N$ is a dense $G_\delta$ subset in $\{ 0,1\}^\N$. For Baire category results concerning the equidistribution mod $1$ of subsequences, the reader may consult \cite{GSW} and the references therein. 
\end{remark}

\section{Characterizations of rigidity sequences}\label{Section 3}

\subsection{Two criteria for rigidity}\label{ss:4a} In this section, we present two criteria for rigidity, namely Theorems~\ref{Theorem 5} and~\ref{vrai?}, which provide a rather tractable way to check the rigidity of a sequence $(n_k)_{k\geq 0}$. The meaning of these results is that if we are able to construct sufficiently many probability measures $\mu $ with ${\much(n_{k})}\rightarrow{1}$, or just sufficiently many measures such that $\overline{\rm lim}_{k\to\infty}\, \vert \widehat\mu(n_k) -1\vert$ is small, then we get ``for free" that there exists a \emph{continuous} probability measure $\mu$ with ${\much(n_{k})}\rightarrow{1}$. And of course, it is presumably much easier to find possibly discrete measures $\mu$ such that $\overline{\rm lim}_{k\to\infty}\, \vert \widehat\mu(n_k) -1\vert$ is small than to construct directly a continuous rigid measure. For more examples of the usefulness of this line of thought, see \mbox{e.g.} \cite{ChoNad}, \cite{Nad} or \cite{Lenz}.

\smallskip
Recall the notation for the set of all rigid measures for a given sequence of integers $(n_k)_{k\geq 0}$:
\[ \mathcal{R}^{(n_k)}=\bigl\{\mu \in\mathcal{P}(\T)\, :\,{\much(n_{k})}\rightarrow{1} \textrm{ as } k\rightarrow\infty\bigr\} .\]

\begin{theorem}\label{Theorem 5}
 Let $\nkp{0}$ be a sequence of integers. Then, $(n_k)$ is a rigidity sequence if and only if $\rnk$ is dense in $\mathcal P(\T)$. Moreover, if this holds, then in fact $\rnk\cap\mathcal P_c(\T)$ is dense in $\mathcal P(\T)$.
\end{theorem}

We point out two consequences of this theorem (see also Corollary~\ref{cor:49} below).  Recall that for any $\mu\in\mathcal P(\T)$, we denote by $G(\mu)$ the subgroup of $\T$ generated by the support of $\mu$, and that if $(n_k)_{k\geq 0}$ is a sequence of integers, then $$E^{(n_k)}=\bigcup\bigl\{ G(\mu);\;\mu\in\rnk\bigr\}.$$
\begin{corollary}\label{points} Let $(n_k)_{k\geq 0}$ be a sequence of integers. Each of the following assertions is equivalent to the rigidity of $(n_k)$.

\begin{itemize}
\item[\rm (a)]  The group $G^{(n_k)}$ generated by $E^{(n_k)}$ is dense in $\T$.
\item[\rm (b)] For any neighborhood $V$ of $1$ in $\T$, there exists a measure $\mu\in\rnk$ such that $\mu\neq\delta_1$ and $\mu(V)>0$.
\end{itemize}
\end{corollary}

 \begin{proof} By Corollary~\ref{densitybis}, (a) is equivalent to the density of $\rnk$. As for (b), it is rather clear that it is implied by the density of $\rnk$; and conversely, if (b) holds true, then it is equally clear that $E^{(n_k)}$ is dense in $\T$.
\end{proof}

This result allows us to retrieve very easily all known examples of rigidity sequences from 
\cite{BDLR,EG,FT,BaGrLyons}. Indeed, the main result of \cite{FT} states that if there exists an irrational $z \in\T$  such that ${z ^{\nk}}\rightarrow{1,}$ then $(n_k)$ is a rigidity sequence; and this follows at once from Corollary \ref{points} (a), since $G(\delta_z)$ alone is already dense in $\T$. Likewise, as already mentioned in Subsection~\ref{ss:3a},  the following generalization of the result from \cite{FT} is proved in \cite[Theorem 2.3]{BaGrLyons}: if the subgroup $$\Gamma^{(n_k)}=\{z \in\T\,;\, {z ^{\nk}}\rightarrow{1}\}$$ is dense in $\T$, then $(n_k)$ is a rigidity sequence; and  again, this follows immediately from Corollary \ref{points} (a). 

\smallskip
However, recall that Fayad and Kanigowski constructed in \cite{FK} examples of rigidity sequences $\nkp{0}$ for which $\Gamma^{(n_k)}=\{ 1\}$ (and that a much stronger result was proved by Griesmer in \cite{Grie}).  
Note that if $\Gamma^{(n_k)}=\{ 1\}$, then $\mathcal{R}^{(\nk)}$ contains only continuous measures except $\delta_1$; so Theorem \ref{Theorem 5} seems useless in this case, \mbox{\it i.e.} the existence of a rigidity sequence $\nkp{0}$ such that 
$\Gamma^{(n_k)}=\{ 1\}$ cannot follow directly from it. Yet, we will see below that a suitable strengthening of Theorem \ref{Theorem 5} can be used to simplify the proof of one of the main results of \cite{FK}.

\smallskip For any sequence of integers $(n_k)_{k\geq 0}$ and any $\varepsilon>0$, let us define
\[
\mathcal{R}_{\varepsilon }^{(\nk)}:=\Bigl\{\mu \in\mathcal{P}(\T)\,:\,\overline{\lim_{k\to\infty}}\,\,|\much(\nk)-1|<\varepsilon\Bigr\}
\]
and $$\Gamma^{(n_k)}_\varepsilon:=\bigl\{ z\in\T;\; \overline{\lim}_{k\to\infty}\, \vert z^{n_k}-1\vert<\varepsilon\bigr\}.$$

\begin{theorem}\label{vrai?} Let $\nkp{0}$ be a sequence of integers. Then  $(n_k)$ is a rigidity sequence if and only if all the sets $\rnk_\varepsilon$, $\varepsilon >0$ are dense in $\mathcal P(\T)$. 
\end{theorem}

From this, we get a stronger version of \cite[Theorem 2.3]{BaGrLyons}.
\begin{corollary}\label{cor:49}  If all the sets $\Gamma^{(n_k)}_\varepsilon$, $\varepsilon >0$ are dense in $\T$, 
then $(n_k)$ is a rigidity sequence.
\end{corollary}
\begin{proof} This is clear since $\rnk_{\varepsilon}$, being convex, contains every measure $\mu\in\mathcal P(\T)$ whose support is finite and contained in $\Gamma^{(n_k)}_\varepsilon$.
\end{proof}

 \subsection{Proof of Theorem \ref{Theorem 5}}
  We have already observed in Corollary \ref{density} that if $\nkp{0}$ is a rigidity sequence, then $\rnk\cap \mathcal P_c(\T)$ is dense in $\mathcal P(\T)$; so we just have to show that if $\rnk$ is dense in $\mathcal P(\T)$, then $(n_k)$ is a rigidity sequence.

  \par\smallskip One possible way to do this is to argue by contradiction and to use \cite[Theorem 2.3]{BaGrLyons}. Indeed, assume that $(n_k)$ is not a rigidity sequence, \mbox{\it i.e.} that $\mathcal R^{(n_k)}$ contains no continuous measure. Since $\mathcal R^{(n_k)}$ is hereditary for absolute continuity, it follows that $\mathcal R^{(n_k)}$ contains only discrete measures. Since $\mathcal R^{(n_k)}$ is dense in $\mathcal P(\T)$ and (again) hereditary for absolute continuity, this implies that $\Gamma^{(n_k)}=\{ z\in\T;\; z^{n_k}\to 1\}=\{ z\in\T;\; \delta_z\in \mathcal R^{(n_k)}\}$ is dense in $\T$. So, by \cite[Theorem 2.3]{BaGrLyons}, $(n_k)$ is a rigidity sequence after all!
  
  \par\smallskip However, since Theorem \ref{Theorem 5} is formally more general than \cite[Theorem 2.3]{BaGrLyons}, it seems desirable to provide a proof which does not rely on the latter. This is what we are going to do now. As a by-product, we will therefore get a new proof of \cite[Theorem 2.3]{BaGrLyons}, which will also be completely different. Indeed, in \cite{BaGrLyons} a continuous measure $\mu\in\mathcal R^{(n_k)}$ was explicitly constructed, while our proof relies on a Baire category argument which can be stated abstractly as follows:

 \begin{lemma}\label{Lemma 6}
  Let $X$ be a Banach space, $P$ a closed convex subset of its bidual $X^{**}$, and $C$ a closed convex subset of $P\cap X$. Let also $(O_{n})_{n\ge 1}$ be a sequence of convex subsets of $P$ which are open in $P$ for the $w^{**}$-$\,$topology. Suppose that 
 $C$ as well as all sets $O_{n}$, $n\ge 1$ are $w^{**}$-$\,$dense in $P$. Then the set
$C\cap\bigcap_{n\geq 1} O_n
$
is norm-dense in $C$.
 \end{lemma}

\begin{proof}
Since $C$ is $w^{**}$-$\,$dense in $P$ and all the sets $O_{n}$ are $w^{**}$-$\,$open in $P$, we see that $C\cap O_{n}$ is $w^{**}$-$\,$dense in $O_{n}$, and hence in $P$. Therefore, $C\cap O_{n}$ is weakly dense in $C$. Since $C$ and $O_{n}$ are convex, it follows that $C\cap O_{n}$ is norm-dense in $C$ by Mazur's theorem. Moreover, $C\cap O_{n}$ is weakly open in $C$, hence norm-open. As $C$ is a closed subset of $X$, the Baire category theorem implies that $\bigcap _{n\ge 1} (C\cap O_{n})$ is norm-dense in $C$, which had to be proved.
\end{proof}

\begin{proof}[Proof of Theorem \ref{Theorem 5}] Going back to the proof of Theorem \ref{Theorem 5}, we assume that $\rnk$ is dense in $\mathcal P(\T)$, and we wish to show that $\mathcal{R}^{(\nk)}\cap\mathcal{P}_{c}(\T)\neq\emptyset$. We will in fact prove directly that  
$\mathcal{R}^{(\nk)}\cap\mathcal{P}_{c}(\T)$ is dense in $\mathcal{P}(\T)$.
\par\smallskip
Observe first that $\mathcal{P}_{c}(\T)$ can be written as $\mathcal{P}_{c}(\T)=\bigcap_{n\ge 1}\mathcal O_{n}$, where each set $\mathcal O_{n}\subseteq \mathcal{P}(\T)$ is open, convex, and dense in $\mathcal{P}(\T)$. (This is a more precise version of the well known fact that $\mathcal P_c(\T)$ is a dense $G_\delta$ subset of $\mathcal P(\T)$.) Indeed, fix for every $k\ge 1$ a finite covering
$(V_{i,k})_{i\in I_{k}}$ of $\T$ by open arcs of length less than $2^{-k}$, in such a way that for every $k\ge 2$ and every $i\in I_{k}$, there exists $i'\in I_{k-1}$ such that $V_{i,k}\subseteq V_{i',k-1}$. For every $n\ge 1$, define
\[
\mathcal{O}_{n}:=\bigl\{\mu \in\mathcal{P}(\T)\,;\,\exists\,k\ \forall\, i\in I_{k}\;:\; \mu (\ba{V}_{i,k})<2^{-n}\bigr\}.
\]
Since the map ${\mu }\mapsto{\mu (F)}$ is upper semi-continuous on $\mathcal{P}(\T)$ for every closed subset $F$ of $\T$, the set $\mathcal{O}_{n}$ is open in $\mathcal{P}(\T)$. Moreover, $\mathcal{O}_{n}$ is convex. Indeed, if $\mu ,\mu'\in \mathcal{O}_{n}$, one can choose $k$ and $k'$ such that $\mu' (\ba{V}_{i,k})<2^{-n}$ for every $i\in I_{k}$ and $\mu' (\ba{V}_{i',k'})<2^{-n}$ for every $i'\in I_{k'}$. Suppose for instance that $k\ge k'$. For every $i\in I_{k}$, there exists an index $i'\in I_{k'}$ such that $\ba{V}_{i,k}\subseteq\ba{V}_{i',k'}$; hence $\mu '(\ba{V}_{i,k})< 2^{-n}$. So, for any $s\in [0,1]$ we have 
$(s\mu +(1-s )\mu ')(\ba{V}_{i,k})<2^{-n}$ for every $i\in I_{k}$, and hence 
$s \mu +(1-s )\mu '$ belongs to $\mathcal{O}_{n}$. It is not difficult to check that $\mathcal{O}_{n}$ is dense in $\mathcal{P}(\T)$ for every $n\geq 1$, and that $\bigcap_{n\ge 1}\mathcal{O}_{n}=\mathcal{P}_{c}(\T)$.
\par\smallskip
Define a map ${J:\mathcal{P}(\T)}\longmapsto{\ell_{\infty}(\N)}$ by setting
\[
(J\mu )(n):=
\begin{cases}
\much(\nk)-1&\textrm{if}\ n=n_k\ \hbox{for some $k$},\\
2^{-n}\much(n)&\textrm{if}\ n\not\in\{\nk\,;\,k\ge 0\}.
\end{cases}
\]
This map $J$ is continuous from $\mathcal{P}(\T)$ into $\ell_{\infty}(\N)=c_{0}(\N)^{**}$ endowed with its $w^{**}$-$\,$topology; and $J$ is also injective. Since $\mathcal{P}(\T)$ is compact, $J$ is an homeomorphism from $\mathcal{P}(\T)$ onto $P:=J(\mathcal{P}(\T))$. In particular, $P$ is $w^{**}$-$\,$closed, hence norm-closed in $\ell_{\infty}( \N )$. Since $J$ is an affine map, $P$ is convex. Set now $C:=J(\mathcal{R}^{(\nk)})$. By the definition of the map $J$, we have $C=P\,\cap\,c_{0}(\N)$, so that $C$ is a closed convex subset of $c_{0}(\N)$. Moreover, $C$ is $w^{**}$-$\,$dense in $P$ since $\rnk$ is dense in $\mathcal P(\T)$. If we set $O_{n}:=J(\mathcal{O}_{n})$ for each $n\ge 1$, the fact that $J$ is an affine homeomorphism implies that $O_{n}$ is convex, $w^{**}$-$\,$open, and $w^{**}$-$\,$dense in $P$. So the hypotheses of Lemma \ref{Lemma 6} are fulfilled, and hence $\bigcap_{n\ge 1}O_{n}\,\cap\,C$ is norm-dense in $C$. In other words, $\bigcap_{n\ge 1} J(\mathcal{O}_{n})\,\cap \,J(\mathcal{R}^{(\nk)})$ is norm-dense in $J(\mathcal{R}^{(\nk)})$, and in particular $w^{**}$-$\,$dense. Since $J$ is a homeomorphism from $\mathcal P(\T)$ onto $(P,w^{**})$, it follows that $\mathcal P_c(\T)\cap \rnk=\bigl(\bigcap_n\mathcal O_n\bigr)\cap\rnk$ is dense in $\rnk$, and hence in $\mathcal P(\T)$. 
\end{proof}

\begin{remark}\label{coupagedecheveux} What the proof of Theorem \ref{Theorem 5} actually shows is the following. Let $(n_k)_{k\ge0}$ be an arbitrary sequence of integers, and let $\mathcal G$ be a $G_\delta$ subset of $\mathcal P(\T)$. Let also $\mathcal P$ be a closed convex subset of $\mathcal P(\T)$. 
Assume that $\mathcal G$ can be written as $\mathcal G=\bigcap_{n\in\N} \mathcal O_n$, where the sets $\mathcal O_n$ are open, {convex} and such that $\mathcal O_n\cap\overline{\rnk\cap\mathcal P}$ is dense in $\overline{\rnk\cap\mathcal P}$. Then $\mathcal G\cap \rnk\cap\mathcal P$ is dense in $\overline{\rnk\cap\mathcal P}$. In particular, if $(n_k)$ is a rigidity sequence, then $\rnk\cap\mathcal G\neq\emptyset$ for any dense $G_\delta$ set $\mathcal G\subseteq\mathcal P(\T)$ which is the intersection of a sequence of convex open sets. Nevertheless, it may be worth pointing out that $\rnk$ is always \emph{meager} in $\mathcal P(\T)$, for any sequence of pairwise distinct integers $(n_k)_{k\geq 0}$. Indeed, if we set $\mathcal F_K:=\bigl\{\mu\in\mathcal P(\T)\,;\; \forall k\geq K\,:\, \vert\widehat\mu(n_k)-1\vert\leq 1/2\bigr\}$ for each $K\in\N$, then $\mathcal F_K$ is a closed set with empty interior in $\mathcal P(\T)$, and $\rnk$ is contained in $\bigcup_{K\in\N}\mathcal F_K$.
\end{remark}


\subsection{Proof of Theorem \ref{vrai?}}
The essential ingredient of the proof is contained in the following claim.

\begin{claim}\label{machin} Let $\varepsilon, \epsilon'>0$. Given any measure $\mu\in\rnk_{\varepsilon}$ and $N\in\N$, $\eta>0$, there exists $\mu'\in \rnk_{\varepsilon'}$ such that 
\begin{enumerate}
 \item[(a)] $\vert \widehat{\mu'}(n)-\widehat\mu(n)\vert<\eta \textrm{ for every } \vert n\vert\leq N$;
\item[(b)] $\sup_{k\geq 0} \vert\widehat{\mu'}(n_k)-\widehat{\mu}(n_k)\vert\leq 4\varepsilon$.
\end{enumerate}
\end{claim}

\begin{proof} Without any loss of generality, we can assume that $\eta, \varepsilon'\leq\varepsilon$. Let $k_0>N$ be such that $\vert\widehat\mu(n_k)-1\vert<\varepsilon$ for all $k> k_0$.  Next, let $M$ be a large integer (how large will be specified at the end of the proof). Since $\rnk_{\varepsilon'}$ is dense in $\mathcal P(\T)$, one can find $\mu_1,\dots ,\mu_M\in\rnk_{\varepsilon'}$ and $k_0<k_1<\cdots <k_M$ such that

\begin{itemize}
\item[\sbt] $\vert \widehat\mu_i(n)-\widehat\mu(n)\vert<\eta$ for every $\vert n\vert\leq n_{k_{i-1}}$;
\item[\sbt] $\vert \widehat\mu_i(n_k)-1\vert<\varepsilon'$ for every $k> k_i$.
\end{itemize}
Now, set
\[\mu':=\frac1M\sum_{i=1}^M \mu_i.\]
Since $\rnk_{\varepsilon'}$ is convex, $\mu'$ belongs to $\rnk_{\varepsilon'}$, and $\vert\widehat{\mu'}(n)-\widehat\mu(n)\vert<\eta$ for every $\vert n\vert\leq N$. For any $k\geq 0$, we have
\[ \vert\widehat{\mu'}(n_k)-\widehat\mu(n_k)\vert\leq \frac1M\sum_{i=1}^M \bigl\vert \widehat{\mu}_i(n_k)-\widehat\mu(n_k)\vert .
\]
If $k\leq k_0$, this gives immediately $\vert\widehat{\mu'}(n_k)-\widehat\mu(n_k)\vert\leq\eta\leq\varepsilon$; and if $k>k_M$, we may write
$ \vert\widehat{\mu}_i(n_k)-\widehat\mu(n_k)\vert =\vert (\widehat{\mu}_i(n_k)-1)-(\widehat\mu(n_k)-1)\vert$ to get that $\vert\widehat{\mu'}(n_k)-\widehat\mu(n_k)\vert\leq \varepsilon+\varepsilon'\leq 2\varepsilon$. Otherwise, there exists $1\leq s\leq M$ such that  $k_{s-1}<k\leq k_{s}$ for some $1\leq s\leq M$. We have in this case
\begin{eqnarray*}
\vert\widehat{\mu'}(n_k)-\widehat\mu(n_k)\vert &\leq& \frac1M\left( \sum_{i=1}^{s-1}\vert(\widehat{\mu}_i(n_k)-1)-(\widehat\mu(n_k)-1)\vert+
\vert \widehat{\mu}_s(n_k)-\widehat\mu(n_k)\vert\right.\\
&+&\left.\sum_{i= s+1}^{M}\vert\widehat{\mu}_i(n_k)-\widehat\mu(n_k)\vert\right)\\
 &\leq&\frac1M\bigl( (s-1)(\varepsilon+\varepsilon')+2+ (M-s)\,\eta\bigr)\\
 &\leq& 3\varepsilon+\frac2M\cdot
\end{eqnarray*}
So $\mu'$ satisfies the required properties (a) and (b) if $M$ is large enough.
\end{proof}

\begin{proof}[Proof of Theorem \ref{vrai?}] By Theorem \ref{Theorem 5}, it is enough to show that $\rnk$ is dense in $\mathcal P(\T)$ under the assumption that all sets $\rnk_{\varepsilon}$ are dense. 
It is therefore enough to show that given $\nu\in\mathcal P(\T)$, $N\in\N$ and $\eta>0$, one can find $\mu\in\rnk$ such that $\vert\widehat\mu(n)-\widehat\nu(n)\vert<\eta$ for all $\vert n\vert\leq N$.
\par\smallskip
Set $\varepsilon_j:= 2^{-j-1}$ for every  $j\geq 1$. By Claim \ref{machin}, one can find a sequence $(\mu_j)_{j\geq 1}$ of elements of $\mathcal P(\T)$ with the following properties:
\begin{enumerate}
\item[(i)] $\mu_j\in\rnk_{\varepsilon_j}$ for every  $j\geq 1$;
\item[(ii)] $\vert\widehat\mu_{1}(n)-\widehat\nu(n)\vert<\eta/2$ for every $\vert n\vert\leq N$;
\item[(iii)] $\vert\widehat\mu_{j+1}(n)-\widehat\mu_j(n)\vert< \varepsilon_j\eta$ for every $j\geq 1$ and every $\vert n\vert\leq N+j$;
\item[(iv)] $\sup_{k\ge 0}| \widehat{\mu}_{j+1}(n_k)-\widehat{\mu}_{j}(n_k)|\leq 4\varepsilon_j$ for every  $j\geq 1$.
\end{enumerate}

\smallskip
By  (iii), the sequence $(\mu_j)_{j\ge 1}$ converges in $\mathcal P(\T)$ to a certain measure $\mu$; and by (ii) and (iii), we have $\vert \widehat\mu(n)-\widehat\nu(n)\vert<\eta$ for all $\vert n\vert\leq N$. Moreover, it follows from (i) and (iv) that $\mu$ belongs to $\rnk$. Indeed we have $\sup_{k\ge 0}| \widehat{\mu}(n_k)-\widehat{\mu}_{r}(n_k)|\leq 4\sum_{j\geq r} \varepsilon_j$ for any $r\geq 1$; so we get
\[ \overline{\lim_{k\to\infty}}\; \vert \widehat\mu(n_k)-1\vert\leq \overline{\lim_{k\to\infty}}\; \vert \widehat\mu_r(n_k)-1\vert+4\sum_{j\geq r}\varepsilon_j\leq \varepsilon_r+4\sum_{j\geq r}\varepsilon_j\qquad\hbox{for any $r\geq 1$}.\]
The proof is now complete.
\end{proof}

\begin{remark}\label{le+economique} In order to show that all sets $\rnk_\varepsilon$ are dense in $\mathcal P(\T)$, it is enough to show that $\delta_z$ belongs to $\overline{\rnk_\varepsilon}$ for any $\varepsilon>0$ and every $z\in\T$. Indeed, the set $$\bigcap_{\varepsilon>0} \overline{\rnk_\varepsilon}$$ is closed and convex; so it is equal to $\mathcal P(\T)$ as soon as it contains every Dirac mass $\delta_z$.
\end{remark}

\subsection{An example} To illustrate Theorem \ref{vrai?}, we show how it can be used to give a streamlined proof of  \cite[Theorem 2]{FK}. 

\begin{example} 
\emph{Let $(n_k)_{k\geq 0}$ be a sequence of integers. Assume that there exists a sequence $(z_m)_{m\geq 1}\subseteq \T$ such that the following properties hold true:
\begin{enumerate}
\item[\rm (a)] for any $z\in\T$, $M\in\N$ and $\alpha>0$, one can find an integer $p$ such that $\vert z_m^p-z\vert<\alpha$ for $m=1,\dots ,M$ (this happens for example if the $z_m$ are rationally independent);
\item[\rm (b)] for any $M\in\N$ and $\eta >0$, one can find $K\geq 0$ such that
\[ \forall k\geq K\;:\; \vert z_m^{n_k}-1\vert<\eta\qquad\hbox{for at least $M-1$ choices of $m\in\{1,\ldots,M\}$}.\]
\end{enumerate}
Then, $(n_k)$ is a rigidity sequence.}
\end{example}
\begin{proof} Let us show that $\rnk_\varepsilon$ is dense in $\mathcal P(\T)$ for any $\varepsilon>0$. By Remark \ref{le+economique}, it is enough to show that $\overline{\rnk_\varepsilon}$ contains the Dirac mass $\delta_z$ for any $\varepsilon>0$ and every $z\in\T$; so let us fix $\varepsilon>0$ and $z\in\T$, and look for a measure $\mu\in\rnk_{\varepsilon}$ which is close to $\delta_z$.

Chose $M$ (depending only on $\varepsilon$) such that $1/M<\varepsilon/4$. Let also $\alpha>0$; choose $p$ (depending on $M$ and $\alpha$) according to property (a) and define 
\[ \mu:= \frac1M\sum_{m=1}^M \delta_{z_m^p}.\]
If $K\geq 0$ satisfies property (b) for $M$ and $\eta:=\varepsilon/2p$, we have 
\[ \vert \widehat\mu(n_k)-1\vert\leq \frac1M\sum_{m=1}^M \vert z_m^{pn_k}-1\vert\leq\frac1M\bigl( (M-1) p\eta+ 2\bigr)\leq3\varepsilon/4 \textrm{ for any }k\geq K.\]
So $\mu$ belongs to $\rnk_{\varepsilon}$. Moreover, since $\vert z_m^p-z\vert<\alpha$ for every $1\le m\le M$, the measure $\mu$ is as close to $\delta_z$ as we wish, provided that $\alpha$ is sufficiently small.
\end{proof}

\subsection{The complexity of rigidity}
Theorem \ref{vrai?} also has a rather unexpected descriptive set-theoretic consequence. Let us denote by $\mathbf{Rig}$ the set of all rigidity sequences $(n_k)_{k\geq 0}$. This is a subset of the Polish space $\Z^{\N_0}$, so it  makes sense to ask for the {descriptive complexity} of $\mathbf{Rig}$. 
By its very definition, $\mathbf{Rig}$ is obviously an analytic set, and it is quite natural to bet that it should be non-Borel. This is however not the case:
\begin{proposition}\label{bizarre} The set $\mathbf{Rig}$ is Borel in $\Z^{\N_0}$; more precisely, it is an $F_{\sigma\delta}$ set. Moreover, $\mathbf{Rig}$ is a \emph{true} $F_{\sigma\delta}$ set, \mbox{\it i.e.} it is not $G_{\delta\sigma}$.
\end{proposition}
\begin{proof} Let $(\mathcal U_q)_{q\in\N}$ be a countable basis of (nonempty) open sets for $\mathcal P(\T)$. By Theorem \ref{vrai?}, for any increasing sequence of integers $\nkp{0}$, we may write
\begin{align*}
(n_k)\in\mathbf{Rig}\iff \forall q\in\N&\;\forall m\in\N\;\,\exists \mu\in\mathcal P(\T)\\
&\Bigl( \mu\in\mathcal U_q\quad\hbox{and}\quad \exists K\in\N\;\forall k\geq K\;:\; \vert\widehat\mu(n_k)-1\vert\leq 2^{-m}\Bigr).
\end{align*}
For each $(q,m)\in\N\times\N$, the relation $R((n_k), \mu)$ under brackets is $F_\sigma$ in $\Z^{\N_0}\times \mathcal P(\T)$; so its projection along the compact factor $\mathcal P(\T)$ is $F_\sigma$ in $\Z^{\N_0}$. This shows that $\mathbf{Rig}$ is an $F_{\sigma\delta}$ subset of $\Z^{\N_0}$. 

\smallskip
To show that $\mathbf{Rig}$ is not $G_{\delta\sigma}$, we use the auxiliary set 
\[ \mathbf W:=\left\{ \alpha\in\N^\N;\; \alpha_i\to\infty\;\,\hbox{as $i\to\infty$}\right\}.\]
It is well known that $\mathbf W$ is a true $F_{\sigma\delta}$ set in $\N^\N$ (see \cite[Section 23]{Ke}). So it is enough to find a continuous map $\Phi:\N^\N\to \Z^{\N_0}$ such that $\Phi^{-1}(\mathbf{Rig})=\mathbf W$. In other words, we are looking for a continuous map $\alpha\mapsto (n_k)_{k\geq 0}$ such that 
\begin{itemize}
\item[\sbt] if $\alpha_i\to\infty$ as $i\to\infty$, then $(n_k)$ is a rigidity sequence;
\item[\sbt] if $\alpha_i\not\to\infty$, then $(n_k)$ is not a rigidity sequence.
\end{itemize}

\smallskip
Let us fix an increasing sequence of integers $(k_i)_{i\in\N}$ with 
\[ k_1=0\qquad{\rm and} \qquad \frac{k_{i+1}}{k_i}\to\infty\quad\hbox{as $i\to\infty$}.\]

\noindent
Given $\alpha\in\N^\N$, we define the sequence $(n_k)_{k\geq 0}$ as follows:
\[ n_0:=1\qquad{\rm and}\qquad n_{k+1}= \alpha_i n_k+k\quad\hbox{for $k_i\leq k<k_{i+1}$},\quad i\in\N.\]

\noindent
The map $\alpha\mapsto (n_k)$ is clearly continuous from $\N^\N$ into $\Z^{\N_0}$. 

If $\alpha_i\to\infty$ as $i\to\infty$, then $\frac{n_{k+1}}{n_k}\to\infty$ as $k\to\infty$, and hence $(n_k)$ is a rigidity sequence. Conversely, assume that $\alpha_i\not\to\infty$ as $i\to\infty$. Then, one can find $q\in\N$ and an increasing sequence $(i_n)_{n\geq 0}\subseteq\N$ such that 
\[ n_{k+1}=qn_k +k\qquad \hbox{for each $n$}\quad\hbox{ and all $k_{i_n}\leq k<k_{i_n+1}$.}\]
By the same arguments as in the proof of \cite[Example 6.4]{BG1}, we see that $(n_k)$ is a Kazhdan sequence, and hence not a rigidity sequence. 
\par\smallskip
For the convenience of the reader, we give a few more details. It is enough to show that if $\varepsilon>0$ is small enough, then condition $(1)_\varepsilon$ in Theorem \ref{Th 1} holds true. Let $\mu\in\mathcal P(\T)$ satisfy $\vert \widehat\mu(n_k)-1\vert<\varepsilon$ for all $k\geq 0$. Then, for each $n$ and all $k_{i_n}\leq k<k_{i_n+1}$, we have
\[ \vert \widehat\mu(k)-1\vert=\vert\widehat\mu(n_{k+1}-qn_k)-1\vert \leq \int_\T \vert z^{n_{k+1}}-1\vert\, d\mu+ q\int_T\vert z^{n_k}-1\vert\, d\mu \leq (1+q)\sqrt{2\varepsilon}.
\]

\noindent So, if we take $\varepsilon$ such that $(1+q)\sqrt{2\varepsilon}\leq 1/2$, we get
\[ \vert \widehat\mu(k)-1\vert\leq \frac12\qquad\hbox{for all $k_{i_n}\leq k<k_{i_n+1}$,}\quad n\geq 0.\]
Since $\frac{k_{i+1}}{k_i}\to\infty$ (which implies that $\frac{k_{i+1}-k_i}{k_{i+1}}\to 1$), it follows that 
\[ \mu(\{ 1\})=\lim_{K\to\infty} \frac1{K+1} \sum_{k=0}^K \widehat\mu(k)\geq \frac12\cdot
\]
This shows that $(n_k)$ is indeed a Kazhdan sequence if $\alpha_i\not\to\infty$; which concludes the proof of Proposition \ref{bizarre}.
\end{proof}

\begin{remark} We find the Borelness of $\mathbf{Rig}$ rather surprising, especially when compared with the following result due to Kaufman \cite{K}. Call a subset $Q$ of $\Z$ a $w$\emph{-set} if there exists a {continuous} complex measure $\mu $ on $\T$ such that
$\inf_{n\in Q}|\much(n)|>0$. Then, the class of $w$-sets is an analytic \emph{non-Borel} subset of $\{0,1\}^{\Z}$.
\end{remark}

\subsection{A question} Recall the notation
\[\Gamma^{(n_k)}_\varepsilon=\Bigl\{ z\in\C\,:\, \overline{\lim_{k\to\infty}}\;\vert z^{n_k}-1\vert<\varepsilon\Bigr\}\]
for a given sequence of integers $(n_k)_{k\geq 0}$.
 By Corollary \ref{cor:49}, we know that $(n_k)$ is a rigidity sequence as soon as all sets $\rnk_\varepsilon$, $\varepsilon>0$ are dense in $\T$. Now, the set $$\Gamma:=\bigcap_{\varepsilon>0}\overline{\Gamma_\varepsilon^{(n_k)}}$$ is easily seen to be a closed subgroup of $\T$. So, in order to show that all sets $\Gamma_\varepsilon^{(n_k)}$ are dense in $\T$, it is enough to check that 
$\Gamma$ is infinite. 

One can also observe that if all the sets $\Gamma_\varepsilon^{(n_k)}$ are infinite, then none of them has isolated points, and hence all the sets $\overline{\Gamma_\varepsilon^{(n_k)}}$ are uncountable. Indeed, let $z\in \Gamma_\varepsilon^{(n_k)}$. Choose $\varepsilon'<\varepsilon$ such that $z\in \Gamma_\varepsilon'^{(n_k)}$, and let $\eta>0$ be such that $2\eta+\varepsilon'<\varepsilon$. Since $\Gamma_\eta^{(n_k)}$ is infinite, one can find a sequence of pairwise distinct points $(a_i)$ in $ \Gamma_\eta^{(n_k)}$ converging to some point $a\in\T$. If we put $b_i:=a_{i+1} \overline{a_i}$, then $b_i$ belongs to $ \Gamma_{2\eta}^{(n_k)}$, $b_i\neq 1$ and $b_i\to 1$. So $z_i:=b_iz$  lies in  $\Gamma_\varepsilon^{(n_k)}$, $z_i\neq g$ and $z_i\to z$. Hence $z$ is not an isolated point of $\Gamma_\varepsilon^{(n_k)}$. This leads to the following question.

\begin{question} Let $(n_k)_{k\geq 0}$ be a sequence of integers. Assume that all the sets $\Gamma_\varepsilon^{(n_k)}$ are infinite or, equivalently, that all the sets $\overline{\Gamma_\varepsilon^{(n_k)}}$ are uncountable. Does it follow that $(n_k)$ is a rigidity sequence?
\end{question}

\section{Rigidity sequences which are dense in the Bohr group}\label{sect:bZ}
\subsection{Density with respect to \emph{some} group topology} Let $(n_k)_{n\ge0}$ be a sequence of integers. By Corollary~\ref{rigid/nullpotent}, if $(n_k)$ is a rigidity sequence, then it is convergent to $0$ with respect to some Hausdorff group topology on $\Z$. Looking for a different behavior of the same sequence, one may ask if there is another group topology on $\Z$ such that $(n_k)$ is \emph{dense} with respect to this new topology. The question of characterizing sequences which are dense with respect to some Hausdorff group topology on $\Z$ has been raised by Ruzsa \cite[p. 478]{R}. The deceptively simple answer is given by the following result.

\begin{proposition}\label{pro:simple}
For \emph{any} sequence of distinct integers $(n_k)_{n\ge0}$, there exists a Hausdorff (even metrizable) group topology $\tau$ of $\Z$ such that $(n_k)$ is dense in $(\Z,\tau)$.
\end{proposition}   
\begin{proof}
According to a classical result of Weyl, the sequence of real numbers $(n_k\theta)_{k\ge0}$ is uniformly distributed mod $1$ for almost all real numbers $\theta$ (with respect to Lebesgue measure). So one can pick an irrational $z\in\T$ such that the sequence $(z^{n_k})_{k\ge0}$ is dense in $\T$. Since $z$ is not a root of $1$, one defines a distance $d$ on $\Z$ by setting
\[ d(n,m):=\vert z^{n}-z^{m}\vert.\]
The distance $d$ is translation-invariant, so the associated topology $\tau$ is a group topology on $\Z$; and it is clear that $(n_k)$ is dense in $(\Z,\tau)$. 
\end{proof}

\subsection{Density with respect to $b\Z$} The refined question we consider now is the existence of rigidity sequences which are dense for the so-called \emph{Bohr topology} of $\Z$. 

\smallskip
Let us denote by $\T_d$ the group $\T$ equipped with the discrete topology. The \emph{Bohr compactification of $\Z$}, denoted by $b\Z$, is the dual group of $\T_d$, \mbox{i.e.} the set of all homomorphisms $\chi:\T\to\T$ endowed with the topology of pointwise convergence. By definition, $b\Z$ is a compact group, and $\Z$ can be viewed as a subgroup of $b\Z$ if we identify an integer $n\in\Z$ with the homomorphism $\chi_n:\T\to\T$ defined by $\chi_n(z)=z^n$. Moreover, it follows from Pontryagin's duality theorem that $\Z$ is dense in $b\Z$ (which explains the terminology). The Bohr topology on $\Z$ is the topology induced by $b\Z$. We refer the reader to \cite{Rudin} for more on the Bohr compactification of a locally compact abelian group.

\par\smallskip
As already mentioned in the introduction, Griesmer proved in \cite{Grie} among other things the following striking result (\cite[Theorem 2.1]{Grie}):

\begin{theorem}\label{Griesmer} There exists a rigidity sequence $(n_k)_{k\geq 0}$ which has the property that every translate of $\{n_k\,;\, k\ge 0\}$ is a  set of recurrence.
\end{theorem}  

It is well known that if $R\subseteq\Z$ is a set of recurrence, then $R$ intersects every Bohr neighborhood of $0$. Hence, if all translates of a set $D\subseteq \Z$ are sets of recurrence, then $D$ is dense in $b\Z$. So Theorem \ref{Griesmer} has the following consequence (\cite[Theorem 8.4]{Grie}): 
\begin{corollary}\label{Bohr}  There exists a rigidity sequence $(n_k)_{k\geq 0}$ which is dense in $b\Z$.
\end{corollary}

In this section, we are going to give a new proof of Corollary \ref{Bohr}, based on ideas from \cite{Katz} and \cite{Gip}. This proof has the advantage of providing an explicit example of a rigidity sequence $(n_k)$ which is dense in $b\Z$. We will also make some comments on Griesmer's proof of Theorem \ref{Griesmer}.

\subsubsection{A new proof of Corollary \ref{Bohr}} 

As explained above, we will explicitly construct a rigidity sequence $(n_k)_{k\geq 0}$ which is dense in $b\Z$. The density of our sequence will be proved by using the following lemma.

\begin{lemma}\label{denseexplicit} Let $(p_j)_{j\geq 1}$ be an increasing sequence of positive integers with the following property: for some sequence $(I_q)_{q\geq 0}$ of pairwise disjoint finite subsets of $\N$ and some constant $c>0$, it holds that 
\[ \forall z\in\T\setminus\{ 1\} \,:\,\inf_{q\geq 0}  \sum_{j\in I_q} \vert z^{p_j}-1\vert>0.\]
Let also $D\subseteq\Z$. Assume that for every $K\in\N$, one can find $q_1,\dots ,q_K\geq 0$ (pairwise distinct) such that $D$ contains the set $\bigl\{ \sum_{j\in F} p_j\,;\; F\subseteq I_{q_1}\cup\cdots \cup I_{q_K}\bigr\}$. Then $D$ is dense in $b\Z$.
\end{lemma}

\smallskip The proof of this lemma rests upon a classical density criterion due to Katznelson (\cite[Theorem 1.3]{Katz}). Recall that $\T$ is the dual group of $\Z$, so one can consider the Fourier transform $\widehat\nu :\T\to\C$ of a probability measure $\nu$ on $\Z$. Explicitly, if $\nu=\sum_{n\in\Z} a_n \delta_{n}$, then 
\[\widehat \nu (z)=\sum_{n\in\Z} a_n z^n\qquad\hbox{for every $z\in\T$}.\]
\begin{proposition}\label{katz}
Let $D\subseteq \Z$. If, for every $\ep > 0$, every $r\ge 1$ and every points $z_1, \cdots, z_r$ in $\T\setminus\{1\}$, there is a probability measure $\nu$ on $\Z$ whose support is contained in $D$ and such that $\vert\wh{\nu}(z_i)\vert < \ep$ for $1\le i\le r$, then $D$ in dense in $b\Z$. 
 \end{proposition}
 
 \smallskip We also need the following elementary fact.
 \begin{fact} \label{S2}
For any finite sequence $(a_j)_{j\in I}\subseteq\T$, the following implication holds: if 
\begin{equation}\label{eq:s1}
\left\vert \prod_{j\in F} a_j- 1\right\vert < \frac43\qquad\hbox{for all $F\subseteq I$},
\end{equation}
then 
 \[ \sum_{j\in I} \left\vert a_j- 1\right\vert \leq \frac{\pi}2\, \left\vert \prod_{j\in I} a_j-1\right\vert.\]
 \end{fact} 
\begin{proof}  First, we note that for every $t\in \R$, we have 
$$ 4 \dist(t,\Z) \le \vert e^{2i\pi t} - 1\vert \le 2\pi \dist(t,\Z).$$
The explicit constants $4$ and $2\pi$ can be obtained using elementary trigonometry and the observation that the sinc
function $\sin(x)/x$ varies between $1$ and $2/\pi$ when $|x| \le \pi/2$. 

Now, suppose that $(a_j)_{j\in I}$ satisfies \eqref{eq:s1}, and write $a_j=e^{2i\pi t_j}$ with $-\frac13\leq t_j<\frac23\cdot$ Then 
\[\dist\left(\sum_{j\in F} t_j, \Z\right) < \frac13\qquad\hbox{for all $F\subseteq I$}.\]
In particular, $-\frac13<t_j<\frac13$ for all $j\in I$. From this, it is easy to deduce that for any $F\subseteq J$, the integer closest to $\sum_{j\in F} t_j$ is $0$. (For example, one can prove it by induction on the cardinality of $F$.) So we have 
$ \sum_{j\in I} \dist(t_j,\Z)=\dist\left(\sum_{j\in I}t _j, \Z\right)$;
and hence 
\[ \sum_{j\in I} \vert a_j-1\vert \leq 2\pi \sum_{j\in I} \dist(t_j,\Z)=2\pi \dist\left(\sum_{j\in I}t _j, \Z\right)\leq\frac\pi2\left\vert \prod_{j\in I} a_j-1\right\vert.\]
\end{proof}

\begin{proof}[Proof of Lemma \ref{denseexplicit}] We apply Proposition \ref{katz}. Let $\ep > 0$ and consider $r\ge 1$ points $z_1, \cdots, z_r$ in $\T\setminus\{1\}$. By assumption, there exists some constant $\gamma>0$ such that 
\[ \sum_{j\in I_q} \vert z_i^{p_j} - 1\vert \geq\gamma\qquad\hbox{for any $q\geq 0$ and $1\leq i\leq r$}.\]
Moreover, we may also assume that $\gamma<4/3$. By Lemma \ref{S2} applied with $a_j:= z_i^{p_j}$, $j\in I_q$ (and since $\frac{2\gamma}\pi\leq\gamma<\frac43$), it follows that for any $q\geq 0$ and $1\leq i\leq r$, one can find a set $F_{i,q}\subseteq I_q$  such that
\[  \vert z_i^{n_{i,q}} - 1\vert \geq \frac{2\gamma}\pi\; ,\qquad\hbox{where} \quad n_{i,q}:=\sum_{j\in F_{i,q}} p_j.\]

Now, let $s$ be a large integer which will be chosen later on, and let $K:=rs$. By assumption, one can find a set $Q\subseteq\Z_+$ with $\#Q=K$ such that $D$ contains the set $\mathscr{S}_Q:=\bigl\{ \sum_{j\in F} p_j\,;\; F\subseteq \bigcup_{q\in Q} I_q\bigr\}$. We enumerate the set $Q$ by $\llbracket 1,s\rrbracket\times \llbracket 1,r\rrbracket$, \mbox{\it i.e.} we write the integers $q\in Q$ as ${q(t,i)}$  with $1\le t\le s$ and $1\le 
i \le r$. 
To each pair $(t,i)\in \llbracket 1,s\rrbracket\times \llbracket 1,r\rrbracket$, we associate the probability measure $\nu_{t,i}$ on $\Z$ defined as follows:
$$ \nu_{t,i} := \frac{1}{2}\bigl(\delta_0 + \delta_{n_{i,q(t,i)}}\bigr).$$
Then
$$ \wh{\,\nu_{t,i}}(z_i) = \frac{1+z_{i}^{n_{i,q(t,i)}}}{2}\cdot$$
Since $\vert z_i^{n_{i,q(t,i)}} - 1\vert \geq \frac{2\gamma}\pi$, it follows (by the parallelogram identity in $\C$; or, to be pedantic, by the formula for the modulus of uniform convexity of the euclidean space $\C$) that 
$$ \vert  \wh{\,\nu_{t,i}}(z_i)\vert \le \left(1 - \left(\frac{\gamma}{\pi}\right)^2\right)^{1/2}.$$
Therefore, the convolution measure 
$$\displaystyle \nu = \Conv_{\;\;\;1\leq i\leq r}\; \Conv_{1\leq t\leq s} \,\, \nu_{t,i}$$
satisfies
$$ \vert\wh{\nu}(z_i)\vert \le  \prod_{t=1}^s \vert \widehat{\,\nu_{t,i}}(z_i)\vert\leq \left(1 - \left(\frac{\gamma}{\pi}\right)^2\right)^{s/2}\qquad\hbox{for all $1\leq i\leq r$}.$$ 
Thus, we have $\vert\wh{\nu}(z_i)\vert < \ep$ for $1\leq i\leq r$ if $s$ is sufficiently large.

To conclude the proof, it remains to observe that the support of $\nu$ is included in $\mathscr{S}_{Q}$, and hence in $D$. 
This shows that $D$ satisfies the criterion stated in 
Proposition~\ref{katz}.
\end{proof}

\smallskip
\begin{example} \emph{Consider the so-called \emph{Erd\"{o}s-Taylor sequence} $(p_j)_{j\geq 1}$ (see \cite{ET}) defined by 
\[ p_1 = 1\qquad{and}\qquad p_{j+1} = jp_j + 1, \quad j\ge 1.\]
For every $q\geq 0$, set $ I_q := ( 2^q,2^{q+1}]$.  Fix also an increasing sequence $0=q_0<q_1<\dots $ such that $N_l:=q_{+1}-q_l\to\infty$, and set $J_{\ell} := \bigcup_{q_l\leq q<q_{l+1}} I_q$. Finally, let $(n_k)_{k\geq 0}$ be the increasing enumeration of the set $D=  \bigcup_{l\geq 0} \mathscr{S}_{\ell}$, where $\mathscr{S}_{\ell}= \left\{ \sum_{j\in F}  p_j\,;\; F\subseteq J_\ell \right\}$. Then, $(n_k)$ is a rigidity sequence which is dense in $b\Z$. }
\end{example}
\begin{proof} The rigidity of $(n_k)$ follows from the fact that the Erd\"{o}s-Taylor sequence $(p_j)$ satisfies 
$$ \sum_{j\ge 1} \left(\frac{p_j}{p_{j+1}}\right)^2 < \infty.$$
By \cite[Theorem 2.3]{Gip}, this implies that $(p_j)$ is \emph{IP-rigid}, which means that there exists a continuous measure $\mu\in\mathcal P(\T)$ such that
\[\widehat\mu \left(\sum_{j\in F} p_j\right)\to 1\qquad\hbox{as $\min(F)\to\infty$}\quad,\quad\hbox{$F\subseteq\N$ finite}.\] 
The rigidity of the sequence $(n_k)$ follows immediately from this property. We refer to \cite{AHL} and \cite{Gip} for more on IP-rigidity.

To prove that $(n_k)$ is dense in $b\Z$, we apply Lemma \ref{denseexplicit}. By the definition of the set $D$, we just need to show that for any 
$z\in\T\setminus\{ 1\}$, there exists some constant $c_z>0$ such that 
\[ \sum_{j\in I_q} \vert z^{p_j}-1\vert\geq c_z\qquad \hbox{for every $q\geq 0$}.\]

Set $\varepsilon=\varepsilon_z:=\frac12\left| z - 1\right|$. By the recurrence relation of the Erd\"os-Taylor sequence, the following implication holds for any $i\geq 1$: 
\[\left| z^{p_i} - 1\right| < \frac{\ep}{i} \implies \left| z^{p_{i+1}} - 1\right| > \ep.\]
(Indeed, we have $\vert z^{p_{i+1}} - 1\vert=\vert z^{1+ip_i}-z^{ip_i}+ z^{ip_i}-1\vert\geq \vert z-1\vert-\vert z^{ip_i}-1\vert\geq \vert z-1\vert -i\,\vert z^{p_i}-1\vert$.)
In particular, we have $\left| z^{p_i} - 1\right| \geq \frac{\ep}{i}$ or $\left| z^{p_{i+1}} - 1\right| \geq\frac{\ep}{i+1}$ for any $i\geq 1$. Therefore, 
if we define
\[ E_q := \Bigl\{ j\in  I_q :  \left| z^{p_j} - 1\right| \ge \frac{\ep}{j}\Bigr\},\]
we see that the cardinality of $E_q$ is at least $\frac12\# I_q$. So we get
\[ \sum_{j\in I_q} \left| z^{p_j} - 1\right| \ge \sum_{j\in E_q} \left| z^{p_j} - 1\right| \ge
\sum_{j\in E_q} \frac{\ep}{j}
    \ge 
    \frac{\ep}{\max (I_q)} \times \frac{\# I_q}{2}
   = \frac{\ep}{2^{q+1}} \times \frac{2^q}{2}
   =  \frac{\ep}{4};
\]
and hence we may take $c_z:=\frac{\varepsilon_z}4\cdot$
\end{proof}

\subsubsection{Comments on Griesmer's proof} In what follows, we denote by $\mathcal R$ the family of all sets $D\subseteq\Z$ such that every translate of $D$ is a set of recurrence.

\smallskip
There are two main steps in Griesmer's proof of Theorem \ref{Griesmer}. The first one is to show that for a large class of measures $\mu\in\mathcal P(\T)$, some sets of integers canonically associated with $\mu$  belong to $\mathcal R$. 
Recall the definition of a \emph{Kronecker set}: a compact set $K\subseteq \T$ is a Kronecker set if every continuous function $f:K\to\T$ can be uniformly approximated by functions of the form $z^n$, $n\in\N$. It is well known that there exist perfect Kronecker sets (see \mbox{e.g.} \cite{KL}).

\begin{proposition}\label{Shkarin} Let $K\subseteq\T$ be an uncountable Kronecker set, and let $\mu\in\mathcal P(\T)$ be a continuous measure supported on $K$. For any $\varepsilon >0$, the set \[ D(\varepsilon,\mu):=\bigl\{ n\in\N\,:\,\vert\widehat\mu(n)-1\vert<\varepsilon\bigr\}\] belongs to $\mathcal R$.
\end{proposition}

This is the most technical part of Griesmer's proof (\cite[Proposition 3.2]{Grie}). The corresponding (weaker) result for Bohr density is due to Saeki \cite{Saeki}, who proved it in the more general context of discrete abelian groups (see also \cite[Section 7.6]{GrahMcG} for a variant of Saeki's proof). 

\smallskip
What we would like to point out here that the density of $D(\varepsilon,\mu)$ for the Bohr topology can also be deduced from a remarkable result of Shkarin \cite{Shk}, a special case of which reads as follows. 
\begin{theorem} Let $X$ be a path-connected, locally path-connected and simply connected topological space, and let $T:X\to X$ be a continuous map which is \emph{minimal}, \mbox{\it i.e.} every $T$-orbit is dense in $X$. Let also $G$ be a compact group which is topologically generated by a single element $g$. Then, for any $x\in X$, the set $\{ (g^n, T^n(x));\; n\in\N\}$ is dense in $G\times X$.
\end{theorem}

We apply this result with $G=b\Z$, which is topologically generated by $g=1\in\Z$, and $X:=L^0(K,\mu,\T)$, the space of all (equivalence classes of) $\mu$-measurable maps $\phi:K\to\T$ endowed with the topology of convergence in $\mu$-measure. The space $X$ is in fact a Polish group, and its topology is the same as that induced by $L^1(K,\mu)$.

Consider the map $T:L^0(K,\mu,\T)\to L^0(K,\mu,\T)$ defined by
\[ T\phi (z)=z\phi(z)\qquad\hbox{for every $\phi\in L^0(K,\mu,\T)$}.\]
The map $T$ is continuous, and since $K$ is a Kronecker set, it is easily checked that $T$ is minimal. Moreover, the space $L^0(K,\mu,\T)$ is \emph{contractible}, and hence path-connected, locally path-connected and simply connected. Indeed, since $\mu$ is a continuous measure, the measure space $(K,\mu)$ is isomorphic to $([0,1[, m)$, where $m$ is Lebesgue measure; so it is enough to show that $L^0([0,1[, m,\T)$ is contractible. Consider the map $H:(t,\phi)\mapsto \phi_t$ from $[0,1]\times L^0([0,1[, m,\T)$ into $L^0([0,1[, m,\T)$ defined as follows: 
\[ \phi_t(x):=\left\{ \begin{matrix} \phi(x)
 &\hbox{if $t\leq x$,}\cr
1 & \hbox{if $t> x$.}\
\end{matrix} \right.\]
The map $H$ is continuous, with $H(0,\phi)=\phi$ and $H(1,\phi)=\mathbf 1$; which proves the contractibility of $L^0([0,1[, m,\T)$.

By Shkarin's Theorem, the set $\{ (n, \mathbf z^n);\; n\in \N\}$ is dense in $b\Z\times L^0(K,\mu,\T)$, where $\mathbf z$ is the function $z\mapsto z$. Since $V:=\left\{ \phi\in L^0(K,\mu,\T)\,:\, \bigl\vert \int (\phi-1)\, d\mu\Bigr\vert<\varepsilon \right\}$ is an open set in $L^0(K,\mu,\T)$ and since $n$ belongs to $ D(\varepsilon,\mu)$ if and only if $\mathbf z^n$ belongs to $ V$, it follows immediately that $D(\varepsilon,\mu)$ is dense in $b\Z$.

\medskip The second main step in Griesmer's proof is to show that one can ``diagonalize" in the family $\mathcal R$ (\cite[Lemma 3.4]{Grie}):
\begin{lemma}\label{diag} Let $(D_s)_{s\ge 1}$ be a sequence of subsets of $\Z$ which is decreasing  with respect to inclusion, and assume that each set $D_s$ belongs to $\mathcal R$. Then, there exists a set $D$ which is almost contained in every $D_s$ (\mbox{\it i.e.} $D\setminus D_s$ is finite for every $s\ge 1$) and still belongs to $\mathcal R$. 
\end{lemma}

Applying this to the sets \[ D_s:=\bigl\{ n\in\N\,:\, \vert\widehat\mu(n)-1\vert<2^{-s}\bigr\},\]
where $\mu$ is a continuous measure supported on a Kronecker sets as in Proposition \ref{Shkarin}, one gets immediately the conclusion of Theorem \ref{Griesmer}: the required sequence $(n_k)_{k\geq 0}$ is the increasing enumeration of the diagonalizing set $D$. 

\medskip As it turns out, the analogue of Lemma \ref{diag} for Bohr density is also true: one can diagonalize in the family of Bohr dense sets. We have not found this result in the literature, so it may after all be new (even though this looks rather surprising). Since this adds no complication, we state it in the general framework of discrete abelian groups. 

\begin{lemma} Let $\mathbf Z$ be a discrete abelian group, and let $(D_s)_{s\ge 1}$ be a decreasing sequence of subsets of $\mathbf Z$. Assume that each set $D_s$ is dense in $b\mathbf Z$. Then, there exists a set $D$ which is almost contained in every $D_s$  and still dense in $b\mathbf Z$. 
\end{lemma}
\begin{proof} We denote by $\mathbf T$ the (compact) dual group of $\mathbf Z$. For any $n\in\mathbf Z$, we denote by $\chi_n\in b\mathbf Z$ the character on $\mathbf T$ defined by $n$. Also, we set $\T^\infty:=\T^\N$ (endowed with the product topology), and we choose a compatible metric $d$ on $\T^\infty$. Finally, for any $n\in\mathbf Z$ and $g=(g_1,g_2,\dots )\in\mathbf T^\N$, set $\chi_n(g):=(\chi_n(g_1),\chi_n(g_2),\dots )\in\T^\infty$.
\begin{claim} Let $D\subseteq\mathbf Z$. Then $D$ is dense in $b\mathbf Z$ if and only if the following holds true:
 \begin{equation}\label{finite} \forall p\in\Z\;\forall\varepsilon>0\;\;\exists F\subseteq D\;\;\hbox{finite}\quad\hbox{such that} \quad\forall g\in\mathbf T^\N\; \exists n\in F \; :\; d(\chi_n(g), \chi_p(g))<\varepsilon.
 \end{equation}
\end{claim}
\begin{proof} Since $\mathbf Z$ is dense in $b\mathbf Z$, the set $D$ is dense in $b\mathbf Z$ if and only if it is dense in $\mathbf Z$ for the Bohr topology. Now, if $p\in\mathbf Z$, a typical neighborhood of $p$ in $\mathbf Z$ for the Bohr topology has the form
\[ \mathcal U(p; g_1,\dots ,g_r, \eta):=\bigl\{ n\in \Z\,:\, \vert\chi_n(g_i)-\chi_p(g_i)\vert<\eta \quad\hbox{for $i=1,\dots ,r$}\bigr\},\]
where $g_1,\dots ,g_r\in\mathbf T$ and $\eta>0$. By definition of the product topology on $\T^\infty$, this means that $p$ has a neighborhood basis made of sets of the form
\[\mathcal V(p; g,\varepsilon):=\bigl\{ n\in\Z\,:\, d(\chi_n(g), \chi_p(g))<\varepsilon\bigr\},\]
where $g\in\mathbf T^\N$ and $\varepsilon>0$. Therefore, $D$ is dense in $\mathbf Z$ for the Bohr topology if and only if
\[ \forall p\in\Z\;\forall\varepsilon>0\; \forall g\in\mathbf T^\N\; \exists n\in D \; :\; d(\chi_n(g), \chi_p(g))<\varepsilon.\]
By compactness of $\mathbf T^\N$ and since the maps $g\mapsto \chi_n(g)$ are continuous from $\mathbf T^\N$ into $\T^\infty$, this is equivalent to (\ref{finite}).
\end{proof}
By the above claim, one can choose for each $s\in\N$ a finite set $F_s\subseteq D_s$ such that
\[ \forall p\in\llbracket -s,s\rrbracket \;\,\forall  g\in\mathbf T^\N \; \exists n\in F_s \; :\; d(\chi_n(g),  \chi_p(g))<2^{-s}.\]
Then, the set $D:=\bigcup_{s\in\N} F_s$ clearly satisfies (\ref{finite}), so it is dense in $b\mathbf Z$; and $D$ is almost contained in every $D_s$ because the sequence $(D_s)$ is decreasing.
\end{proof}

\section{Some open questions}\label{Section 5}
Some interesting open problems related to rigidity and \ka\ sequences concern the so-called \emph{Furstenberg sequence} $\nkp{0}$ obtained by ordering in a strictly increasing fashion the elements of the non-lacunary multiplicative semigroup $\{2^{i}3^{j}\,;\,i,j\ge 0\}$. A first question from \cite{BDLR} concerns rigidity.

\begin{question}\label{Question 11}
 Is the Furstenberg sequence a rigidity sequence? 
\end{question}
It is proved in \cite[Theorem\,2]{R} (see also \cite{Nathan}) that the Furstenberg sequence is nullpotent. In particular, the Furstenberg set is not an asymptotic basis of $\Z$.

\par\smallskip
The next question appeared in \cite{BaGrLyons}.

\begin{question}\label{Question 12}
 Is the Furstenberg sequence a \ka\ sequence?
\end{question}

In order to show that the Furstenberg set is not \ka, it would suffice to prove that there exists for every $\delta \in(0,1)$ a measure $\mu \in\mathcal{P}_{c}(\T)$ with $\inf_{i,j\ge 0}|\much(2^{i}3^{j})|\ge \delta $. The existence of such a measure is proved in \cite{BaGrLyons} for every $\delta \in(0,1/2)$. Of course, one cannot have positive answers to both Questions~\ref{Question 11} and ~\ref{Question 12}.

\par\smallskip
The following question from \cite{BDLR} is also open.

\begin{question}\label{Question 11b}
Is the sequence $n_k = 2^k + 3^k$ rigid ?
\end{question}

\par\smallskip

In a completely different direction, we propose the following question. Let $\mathcal{K}$ be the family of \ka\ subsets of $\Z$, seen as a subset of $\{0,1\}^{\Z}$ (endowed with its natural product topology), and let $\neg\mathcal{K}_{0}$ be the family of {generating} non-\ka\ subsets of $\Z$.

\begin{question}\label{Question 13}
 Are the classes $\mathcal{K}$ and $\neg\mathcal{K}_{0}$ Borel in $\{ 0,1\}^\Z$?
\end{question}

A \emph{negative} answer to this question would in particular imply something much stronger than the existence of \ka\ subsets of $\Z$ which are not asymptotic bases. Indeed, the set of all asymptotic bases of $\Z$ is easily seen to be Borel in $\{ 0,1\}^\Z$ (more precisely, $G_{\delta\sigma}$); so the non-Borelness of $\mathcal K$ would say that the properties of being a Kazhdan set and that of  being an asymptotic basis are in fact extremely different.

\smallskip 
Recall that the corresponding question for the set of rigid sequences has a positive answer by Corollary \ref{bizarre}. However, we have been unable to solve the following related question.

\begin{question} Consider the set of all sequences of integers $\nkp{0}$ for which there exists an irrational $z\in\T$ such that $z^{n_k}\to 1$. Is this set Borel in $\Z^{\N_0}$?
\end{question}

\end{document}